\documentclass[a4paper,11pt]{article}

\usepackage{amsthm,amsmath,amssymb,nicefrac,bbm,epsfig,enumerate,geometry}
\usepackage{xcolor}
\usepackage[hidelinks,colorlinks=false,linkcolor=false,urlcolor=false]{hyperref}
\usepackage[pagewise]{lineno}[4.41]
\usepackage[capitalise,nameinlink,noabbrev]{cleveref}
\newtheorem{lemma}{Lemma}[section]
\newtheorem{remark}[lemma]{Remark}
\newtheorem{proposition}[lemma]{Proposition}
\newtheorem{theorem}[lemma]{Theorem}

\newtheorem{corollary}[lemma]{Corollary}
\newtheorem{setting}[lemma]{Setting}

\newcommand{\1}{\mathbbm{1}}
\providecommand{\F}{{\ensuremath{\mathbbm{F}}}}
\providecommand{\N}{{\ensuremath{\mathbbm{N}}}}
\providecommand{\Z}{{\ensuremath{\mathbbm{Z}}}}
\providecommand{\R}{{\ensuremath{\mathbbm{R}}}}

\providecommand{\E}{{\ensuremath{\mathbb{E}}}}
\renewcommand{\P}{{\ensuremath{\mathbb{P}}}}
\newcommand{\var}{\mathrm{Var}}

\newcommand{\epsUp}{\overline{\varepsilon}}
\newcommand{\epsDown}{\underline{\varepsilon}}

\newcommand{\uniform}{\ensuremath{\mathcal{R}}}
\newcommand{\unif}{\ensuremath{\mathfrak{r}}}

\allowdisplaybreaks
\newcommand{\vastleft}[2]{\left#2 \rule{0pt}{#1}\kern-.25ex\right.}
\newcommand{\vastright}[2]{\left. \rule{0pt}{#1}\kern-.25ex \right#2}
\renewcommand{\gets}{\curvearrowleft}

\title{
Multilevel Picard approximations for\\
McKean-Vlasov stochastic differential equations}

\author{Martin Hutzenthaler\thanks{Faculty of Mathematics, University of Duisburg-Essen, Essen, Germany;\newline\hspace*{.5cm} e-mail: \texttt{martin.hutzenthaler}\textcircled{\texttt{a}}\texttt{uni-due.de}}
\and
\and 
Tuan Anh Nguyen\thanks{Faculty of Mathematics, University of Bielefeld, Bielefeld, Germany; \newline\hspace*{.5cm} e-mail: \texttt{tnguyen}\textcircled{\texttt{a}}\texttt{math.uni-bielefeld.de}}
}

\begin{document}\sloppy

\title{
Full history recursive multilevel Picard approximations
suffer
from\\
the curse of dimensionality
for the\\
Hamilton-Jacobi-Bellman equation of
a stochastic control problem}
\maketitle
{
\let\thefootnote\relax
\let\thefootnote\relax
}
\begin{abstract}
  Full history recursive multilevel Picard (MLP) approximations
  have been proved to overcome the curse of dimensionality
  in the numerical approximation of semilinear heat equations
  with nonlinearities which are globally Lipschitz continuous
  with respect to the maximum-norm.
  Nonlinearities in
  Hamilton-Jacobi-Bellman equations in stochastic control theory,
  however, are often (locally) Lipschitz continuous with respect
  to the standard Euclidean norm.
  In this paper we prove the surprising fact
  that MLP approximations
  for one such example equation suffer from the curse of dimensionality.
\end{abstract}
\tableofcontents
\section{Introduction}
  High-dimensional
  second-order
  partial differential equations (PDEs)
  are abundant in many important areas;
  see e.g.\ the surveys~\cite{ElKarouiPengQuenez1997,BachmayrSchneiderUschmajew2016}.
  In this article, we focus on Hamilton-Jacobi-Bellman (HJB) equations
  in stochastic control problems
  in which, e.g., in mean field games the dimension
  is related to the number of agents;
  see, e.g., \cite{LasryLions2006a,LasryLions2006b,HuangMalhameCaines2006}.
  It is often unclear whether such problems suffer from the curse of dimensionality.

  To illustrate our results we consider in this paragraph the following classical
  Gaussian
  control problem in which we control the drift
  of all components of a Brownian motion except for one.
  Let $(\Omega,\mathcal{F},\P)$ be a probability space, let $d\in\N=\{1,2,\ldots\}$,
  let $W\colon[0,1]\times\Omega\to\R^d$ be a standard Wiener process,
  let $\mathbb{F}=(\mathbb{F}_t)_{t\in[0,1]}$ be the complete and
  right-continuous filtration which is generated by $W$,
  for every $s\in[0,1]$, $x\in\R^d$, and every
  $(\mathbb{F}_t)_{t\in[s,1]}$-adapted and
  measurable function $C\colon[s,T]\times\Omega\to\R^d$
  with
  $\sup_{s\in[0,1]}\|C_s\|\leq 1$
  let
  $X_{s,\cdot}^{x,C}\colon[s,1]\times\Omega\to\R^d$ satisfy
  for all $t\in[s,1]$ that
  \begin{equation}  \begin{split}\label{eq:controlled.SDE}
    X_{s,t}^{x,C}=x+\int_s^t C_r\,dr+W_t-W_s,
  \end{split}     \end{equation}
  let  $u\colon[0,1]\times\R^d\to\R$
  satisfy for all $s\in[0,1]$, $x\in\R^d$ that
  \begin{equation}  \begin{split}\label{eq:intro.stochastic.control}
    u(s,x)=\sup_{\substack{C\colon[s,1]\times\Omega\to\R^d
    \text{ measurable and $(\mathbb{F}_t)_{t\in[s,1]}$-adapted}\\
    \sup_{s\in[0,1]}\|C_s\|\leq 1,\;
    (1,0,\ldots,0)C\equiv0}}\E\Big[
    \big|X_{s,1}^{x,C}(1)\big|\Big].
  \end{split}     \end{equation}
 Bellman's seminal dynamic programming principle shows 
  that the value function $u$ satisfies the 
  HJB equation, that is, for all $s\in[0,1]$, $x=(x_1,\ldots,x_d)\in\R^d$ that
  $u(1,x)=|x_1|$ and
  \begin{equation}  \begin{split}\label{eq:semilinear.heat}
    &\tfrac{\partial}{\partial s}u(s,x)
    +\tfrac{1}{2}\Delta_xu(s,x)
    +\Big(\sum_{j=2}^d\big|(\tfrac{\partial}{\partial x_j}u)(s,x)\big|^2\Big)^{\frac{1}{2}}
\\&
=
    \tfrac{\partial}{\partial s}u(s,x)
    +\tfrac{1}{2}\Delta_xu(s,x)
    +\sup_{c=(c_1,\ldots,c_d)\in\R^d\colon c_1=0,\|c\|\leq 1}
    \Big[\langle \nabla_xu(s,x),c\rangle
    \Big]
    =0;
  \end{split}     \end{equation}
  cf., e.g., \cite[Proposition 4.3.5]{YongZhou1999_Stochastic_controls}.
  It is not difficult to see that
  the value function satisfies for
all $s\in[0,1]$, $x=(x_1,\ldots,x_d)\in\R^d$ that
$u(s,x)=\E\big[|x_1+W_1(1)-W_s(1)|\big]$.

The only known approximation method in the literature which provably
does not
suffer from the curse of dimensionality in the numerical approximation
of semilinear heat equations with Lipschitz continuous and
gradient-independent nonlinearities is the
full history multilevel Picard approximation method (MLP)
introduced \cite{e2021multilevel} and improved in \cite{hutzenthaler2020overcomingDefault,hutzenthaler2020multilevelXglob,hutzenthaler2021overcoming};
cf.\ also \cite{beck2020overview} and \cite{EHutzenthalerJentzenKruse2017}
for comparisons with other
approximation methods.
For further results on this method see, e.g.,
\cite{beck2024overcoming,becker2020numerical,hutzenthaler2020multilevel,hutzenthaler2020proof,hutzenthaler2020overcoming,hutzenthaler2022multilevel,hutzenthaler2022multilevelXnonglob,neufeld2023multilevel,neufeld2025multilevel}.

The main contribution of this paper is to show via a counterexample
that the MLP approximation method in general does not overcome the curse
of dimensionality for stochastic control problems in the sense that
the approximation error might grow faster in the dimension than any polynomial.
This is surprising for several reasons. First, the target function in our
counterexample depends only on the first spatial coordinate so that the
approximation problem for this target function does not suffer from the
curse of dimensionality.
Second, the approximation errors between this solution
and  fixed-point iterations
with respect
to a Feynman-Kac-type fix-point equation
grow at most polynomially in the dimension; see \cite[Proposition 4.1]{hutzenthaler2021speed}.
Third, the approximation errors
of the MLP approximation method
were shown to
grow at most polynomially in the dimension if the nonlinearity
is globally Lipschitz continuous
with respect to the maximum-norm and the Lipschitz constant is counded in the dimension;
see \cite{hutzenthaler2021overcoming,neufeld2025multilevel}.
In fact, this latter
condition is not satisfied by the nonlinearity
 $\R^d\ni(z_1,\ldots,z_d)\mapsto (\sum_{j=2}^d|z_j|^2)^{1/2}\in\R$
since
\begin{equation}  \begin{split}
  \sup_{x=(x_1,\ldots,x_d),y=(y_1,\ldots,y_d)\in\R^d\colon x\neq y}
  \tfrac{\Big|\big(\sum_{j=2}^d|x_j|^2\big)^{1/2}
-\big(\sum_{j=2}^d|y_j|^2\big)^{1/2}\Big|}
{\max_{j\in\{1,\ldots,d\}}|x_j-y_j|}=\sqrt{d-1}
\end{split}     \end{equation}
is unbounded in the dimension $d$.
%
%
%
It turns out, that the latter condition on the Lipschitz constants
is essential at least to some extent.
More precisely, the following theorem, Theorem \ref{thm:intro}, shows
via a counterexample that is not enough that the nonlinearity is globally Lipschitz
continuous and that the Lipschitz constant is bounded in the dimension.

\begin{theorem}\label{thm:intro}
Let 
$\Theta=\bigcup_{n\in\N}\Z^n$,
for every $d\in\N$ let 
$\lVert\cdot \rVert_{\R^d},f_d,g_d\colon \R^d\to\R$,
$u_d \colon [0,1]\times\R^d\to \R$
 satisfy
for all $t\in[0,1)$, $x=(x_1,x_2,\ldots,x_d)\in\R^d$ that 
$
 \lVert x \rVert_{\R^d}= \bigl({\textstyle\sum_{j=1}^{d}}\lvert x_j\rvert^2\bigr)^{1/2}$,
$
g_d(x)=\lvert x_1\rvert$, 
$ f_d(x)=\bigl({\textstyle\sum_{j=2}^{d}}\lvert x_j\rvert^2\bigr)^{1/2}$,
$u_d|_{[0,1)\times \R^d}\in C^{1,2}([0,1)\times \R^d,\R)$,
\begin{align}\label{eq:intro.PDE}
\bigl( \tfrac{ \partial }{ \partial t } u_d \bigr)( t, x ) + \tfrac{ 1 }{ 2 } ( \Delta_x u_d )( t, x ) + (f_d\circ \nabla_x u_d)  (t, x)  = 0, \quad\text{and}\quad
u_d(1,x)=g_d(x)
,
\end{align}
let $(\Omega, \mathcal{F},\P)$ be a probability space,
let $\unif^\theta\colon \Omega\to (0,1] $, $\theta\in\Theta$,
be i.i.d.\ random variables which satisfy for all $b\in [0,1] $ that 
$\P(\unif^0\leq b)=\sqrt{b}$, 
for every $\theta\in \Theta$
let 
$\uniform^\theta=(\uniform^\theta_t)_{t\in[0,1]}\colon [0,1]\times\Omega\to [0,1] $ satisfy for all  $t\in [0,1]$ that
$\uniform^\theta_t=t+(1-t)\unif^\theta$,
let $W^{d,\theta}\colon [0,1]\times \Omega \to\R^d$, $\theta\in\Theta$, 
$d\in \N$,
be independent standard 
Brownian motions with continuous sample paths,
assume for every $d\in\N$ that
$(W^{d,\theta})_{\theta\in\Theta}$ and
$(\unif^\theta)_{\theta\in\Theta}$ are independent,
and
let
$ 
  {\bf U}_{ n,m}^{d,\theta }=
({U}_{ n,m}^{d,\theta},{V}_{ n,m}^{d,\theta})
  \colon[0,1)\times\R^d\times\Omega\to\R^{1}\times\R^d
$,
$n,m\in\Z$, $\theta\in\Theta$, 
$d\in\N$,
satisfy
for all 
$d,
  n,m \in \N
$,
$ \theta \in \Theta $,
$ t\in [0,1)$,
$x \in \R^d$
that $
{\bf U}_{-1,m}^{\theta}(t,x)={\bf U}_{0,m}^{\theta}(t,x)=0$ and 
\begin{equation}  \begin{split}\label{eq:intro.MLP}
  &{\bf U}_{n,m}^{d,\theta}(t,x) 
  =  (g_d(x),0)+
  \sum_{i=1}^{m^n}\tfrac{g_d\bigl(x+W^{d,(\theta,0,-i)}_1-W^{d,(\theta,0,-i)}_t\bigr)-g_d(x)}{m^n}
  \left(
  1 ,
  \tfrac{ 
  W^{d,(\theta, 0, -i)}_{1}- W^{d,(\theta, 0, -i)}_{t}
  }{ 1 - t }
  \right)
  \\
  &+\sum_{\ell =0}^{n-1}\sum_{i=1}^{m^{n-\ell }}
  \tfrac{
\left((f_d\circ 
  {V}_{\ell ,m}^{d,(\theta,\ell ,i)})-\1_{\N}(\ell ) (f_d\circ {V}_{\ell -1,m}^{d,(\theta,-\ell ,i)})\right)
\!
  \left(\mathcal{R}^{(\theta, \ell ,i)}_t,
x+W_{\mathcal{R}^{(\theta, \ell ,i)}_t}^{d,(\theta,\ell ,i)}-W_t^{d,(\theta,\ell ,i)}\right)
  }
  {(m^{n-\ell })/\big(4(1-t)\big(\mathcal{R}^{(\theta, \ell ,i)}_t-t\big)\big)^{\frac{1}{2}}}
  \left(
  1 ,
  \tfrac{ 
  W_{\mathcal{R}^{(\theta, \ell ,i)}_t}^{d,(\theta,\ell ,i)}- W^{d,(\theta, \ell , i)}_{t}
  }{ \mathcal{R}^{(\theta, \ell ,i)}_t-t}
  \right).
\end{split}    \end{equation}
Then 
 for all $p\in[0,\infty)$, $n\in\N\cap(2p,\infty)$
it holds that
\begin{align}
\liminf_{d\to\infty}\frac{1}{d^{p}}
\left(
\E\!\left[
\left\lVert{\bf U}_{n,n}^{d,0}(0,0)-(u_d,\nabla_x u_d)(0,0)\right \rVert^2_{\R^{d+1}}\right]\right)^{1/2}=\infty.
\end{align}

\end{theorem}
Theorem \ref{thm:intro} follows directly from Corollary \ref{c12} together with
uniqueness of the solution of the PDE \eqref{eq:intro.PDE}.
Theorem \ref{thm:intro} shows that the $L^2$-distance between
the solution of the PDE  \eqref{eq:intro.PDE} and the MLP approximations \eqref{eq:intro.MLP}
grows faster in the dimension than any polynomial.

\section{Upper bounds for MLP approximations with nonlinearities of seminorm-type}
In this section we consider nonlinearities which are of seminorm-type.
By this we mean functions $f\colon\R^d\to\R$ such that $|f|$
satisfies condition \eqref{p01}.
Under this condition we derive upper bounds for second moments of our
MLP approximations using repeated applications of the triangle-type inequality \eqref{p01}.
%
The following setting introduces MLP approximations.
\begin{setting}\label{s01}
Let $d\in\N$, 
$\Theta=\bigcup_{n\in\N}\Z^n$,
$f,g\in C(\R^d,\R)$,
let $\varrho \colon  [0,1]^2\to \R$ satisfy for all
$t\in[0,1)$, 
$s\in (t,1]$ that $\varrho (t,s)= \frac{1}{2\sqrt{1-t}\sqrt{s-t}}$,
let $(\Omega, \mathcal{F},\P,(\F_t)_{t\in[0,1]})$ be a filtered probability space satisfying the usual conditions,
let $\unif^\theta\colon \Omega\to (0,1] $, $\theta\in\Theta$, be i.i.d.\ random variables which satisfy for all $b\in [0,1] $ that 
$\P(\unif^0\leq b)=\int_{0}^{b}\varrho (0,s)\,ds$, 
for every $\theta\in \Theta$
let 
$\uniform^\theta=(\uniform^\theta_t)_{t\in[0,1]}\colon [0,1]\times\Omega\to [0,1] $ satisfy for all  $t\in [0,1]$ that
$\uniform^\theta_t=t+(1-t)\unif^\theta$,
let $W^\theta\colon [0,1]\times \Omega \to\R^d$, $\theta\in\Theta$, be independent standard 
$(\F_t)_{t\in[0,1]}$-Brownian motions with continuous sample paths,
assume that
$(W^\theta)_{\theta\in\Theta}$ and
$(\unif^\theta)_{\theta\in\Theta}$ are independent,
and
let
$ 
  {\bf U}_{ n,m}^{\theta }=
({U}_{ n,m}^{\theta},{V}_{ n,m}^{\theta})
  \colon[0,1)\times\R^d\times\Omega\to\R^{1}\times\R^d
$,
$n,m\in\Z$, $\theta\in\Theta$, satisfy
for all 
$
  n,m \in \N
$,
$ \theta \in \Theta $,
$ t\in [0,1)$,
$x \in \R^d$
that $
{\bf U}_{-1,m}^{\theta}(t,x)={\bf U}_{0,m}^{\theta}(t,x)=0$ and 
\begin{equation}  \begin{split}
  &{\bf U}_{n,m}^{\theta}(t,x) 
  =  (g(x),0)+
  \sum_{i=1}^{m^n}\tfrac{g\bigl(x+W^{(\theta,0,-i)}_1-W^{(\theta,0,-i)}_t\bigr)-g(x)}{m^n}
  \left(
  1 ,
  \tfrac{ 
  W^{(\theta, 0, -i)}_{1}- W^{(\theta, 0, -i)}_{t}
  }{ 1 - t }
  \right)
  \\
  &+\sum_{\ell =0}^{n-1}\sum_{i=1}^{m^{n-\ell }}
  \tfrac{
\left((f\circ 
  {V}_{\ell ,m}^{(\theta,\ell ,i)})-\1_{\N}(\ell ) (f\circ {V}_{\ell -1,m}^{(\theta,-\ell ,i)})\right)
\!
  \left(\mathcal{R}^{(\theta, \ell ,i)}_t,x+W_{\mathcal{R}^{(\theta, \ell ,i)}_t}^{(\theta,\ell ,i)}-W_t^{(\theta,\ell ,i)}\right)
  }
  {m^{n-\ell }\varrho(t,\mathcal{R}^{(\theta, \ell ,i)}_t)}
  \left(
  1 ,
  \tfrac{ 
  W_{\mathcal{R}^{(\theta, \ell ,i)}_t}^{(\theta,\ell ,i)}- W^{(\theta, \ell , i)}_{t}
  }{ \mathcal{R}^{(\theta, \ell ,i)}_t-t}
  \right).
\end{split}  \label{c10}   \end{equation}
\end{setting}
\begin{remark}[Densities of $\uniform^0_t$, $t\in[0,1)$]\label{d14}  Assume \cref{s01}. 
Then
the fact that
$\forall\,t \in [0,1]\colon \uniform_t^0=t+(1-t)\unif^0$,
the fact that
$\forall\, b\in [0,1] \colon \P(\unif^0\leq b)=\int_{0}^{b}\varrho (0,r)\,dr=\int_{0}^{b}
\frac{dr}{2\sqrt{r}}$, and
the substitution rule 
show for all $t\in [0,1)$, $h\in C([t,1],[0,\infty))$ that
\begin{align}\begin{split}
&
\E \!\left[h(\uniform_t^0)\right]
= \E \!\left[h (t+(1-t)\unif^0)\right]
=
\int_{0}^{1}
\frac{h (t+(1-t)r)}{2\sqrt{r}}\,dr=\int_{t}^{1} \frac{h(s)}{2\sqrt{\frac{s-t}{1-t}}}\frac{ds}{1-t}
= \int_{t}^{1}h(s)\varrho(t,s)\,ds. 
\end{split}\end{align}
\end{remark}
\begin{lemma}\label{c02}
It holds for all $t_0\in [0,1)$, $k\in\N$ that
\begin{align}
\int_{t_0}^1\int_{t_1}^{1}\ldots\int_{t_{k-1}}^{1}
\prod_{j=1}^{k}
\frac{\sqrt{1-t_{j-1}}}{\sqrt{t_j-t_{j-1}}}
dt_{k}dt_{k-1}\cdots dt_{1}\leq 2^k.
\end{align}
\end{lemma}
\begin{proof}[Proof of \cref{c02}]
First, observe for all $t\in[0,1)$ that
\begin{align}
\int_{t}^{1}\frac{\sqrt{1-t}}{\sqrt{s-t}}\,ds= 2\sqrt{1-t}\sqrt{s-t}\bigr|_{s=t}^1=2(1-t)\leq 2 .\label{t23} 
\end{align}
This shows 
 for all $k\in\N$, $t_0\in [0,1)$  that
\begin{align}\begin{split}
&
\int_{t_0}^1\int_{t_1}^{1}\ldots\int_{t_{k-1}}^{1}\int_{t_k}^{1}
\prod_{j=1}^{k+1}
\frac{\sqrt{1-t_{j-1}}}{\sqrt{t_j-t_{j-1}}}\,
dt_{k+1}dt_{k}dt_{k-1}\cdots dt_{1}\\
&
= 
\int_{t_0}^1\int_{t_1}^{1}\ldots\int_{t_{k-1}}^{1}\left[\int_{t_k}^{1}
\frac{\sqrt{1-t_{k}}}{\sqrt{t_{k+1}-t_{k}}}\,
dt_{k+1}\right]
\prod_{j=1}^{k}
\frac{\sqrt{1-t_{j-1}}}{\sqrt{t_j-t_{j-1}}}\,
dt_{k}dt_{k-1}\cdots dt_{1}\\
&
\leq 
\int_{t_0}^1\int_{t_1}^{1}\ldots\int_{t_{k-1}}^{1}2
\prod_{j=1}^{k}
\frac{\sqrt{1-t_{j-1}}}{\sqrt{t_j-t_{j-1}}}\,
dt_{k}dt_{k-1}\cdots dt_{1}.
\end{split}\end{align}
This, \eqref{t23}, and induction complete the proof of \cref{c02}.
\end{proof}

Under the triangle-type inequality \eqref{p01} below we now derive
upper bounds for second moments.
\begin{proposition}\label{m01}
Assume \cref{s01},
let $\lVert\cdot\rVert\colon\R^d\to[0,\infty)$ be the standard norm, let
$\gamma,f,g\in C(\R^d,\R)$ 
satisfy
for all $\lambda\in\R $, 
$v_1,v_2\in \R^d$ that 
\begin{align}\label{c01}
\lvert g(v_1)-g(v_2)\rvert
\leq  \gamma (v_1-v_2) ,
\quad \gamma(v_1+v_2)\leq \gamma (v_1)+\gamma(v_2),\quad \gamma(\lambda v_1)\leq  \lvert \lambda \rvert\gamma (v_1),
\end{align}
\begin{align}
\lvert f(v_1+v_2)\rvert\leq \lvert  f(v_1)\rvert+\lvert f(v_2)\rvert, \quad\text{and}\quad  \lvert f(\lambda v_1)\rvert\leq \lvert\lambda f(v_1)\rvert.\label{p01}
\end{align}
Then \begin{enumerate}[i)]
\item \label{c17}it holds for all 
$n,m\in\N$, $\theta\in\Theta$ that
$\mathbf{U}_{n,m}^\theta$ is measurable,
\item\label{c13}  it holds for all
$n,m\in\N$,
$t\in[0,1)$, $\xi\in\R^d$  that
\begin{align}\begin{split}
&
\left(\E\!\left[
\left\lvert
\left(f
\circ 
V_{n,m}^0\right)(t,\xi)
\right\rvert^2\right]
\right)^{1/2}\leq \left(\E\!\left[\lvert \gamma( W^0_1) \rvert^4\right]\right)^{1/4}
 \max\left\{\left(\E\!\left[\lvert f( W^0_1)\rvert^4\right]\right)^{n/4},1\right\}6^{n-1}<\infty,
\end{split}\end{align}
and
\item \label{c14}it holds for all 
$n,m\in\N$ that
\begin{align}\begin{split}
&
\left(
\E\!\left[
\left\lvert{U}_{n,m}^{0}(0,0)\right \rvert^2\right]+
\E\!\left[
\left\lVert{V}_{n,m}^{0}(0,0)\right \rVert^2\right]\right)^{1/2}
\\
&
\leq \lvert g(0)\rvert+
\left(\E\!\left[\lvert \gamma( W^0_1) \rvert^4\right]\right)^{1/4}
 \max\left\{\left(\E\!\left[\lvert f( W^0_1)\rvert^4\right]\right)^{n/4},1\right\}6^n\sqrt{d}<\infty
.
\end{split}\end{align}
\end{enumerate}
\end{proposition}

\begin{proof}[Proof of \cref{m01}]
First, \cite[Lemma 3.2]{hutzenthaler2021overcoming} proves \eqref{c17}.
Throughout the rest of this proof
let $c_1,c_2\in [0,\infty]$ satisfy that 
\begin{align}\label{t15}
c_1=\left(\E\!\left[\lvert f( W^0_1) \rvert^4\right]\right)^{1/4}\quad\text{and}\quad
c_2=  \left(\E\!\left[\lvert \gamma( W^0_1) \rvert^4\right]\right)^{1/4}
\end{align}
and
for every $N,m\in\N $  let $\epsUp_{N,m}\colon \{0,1,\ldots,N-1\}\to[0,\infty]$ satisfy for all $n\in \{0,1,\ldots,N-1\} $ that
\begin{align}\small\begin{split}
&
\epsUp_{N,m}(n)=\sup_{\substack{k\in\N\colon k+n\leq N}}\sup_{t_0\in[0,1)}
\sup_{\xi\in\R^d}\\
&
\left(\int_{t_0}^1\int_{t_1}^{1}\ldots\int_{t_{k-1}}^{1}
\E\!\left[
\left\lvert
\left(f
\circ 
V_{n,m}^0\right)(t_k,\xi + W^0_{t_k}-W_{t_0}^0)
\prod_{j=1}^{k}
\tfrac{f(W_{t_j}^0-W^0_{t_{j-1}})}{\sqrt{\varrho(t_{j-1},t_j)}(t_{j}-t_{j-1})}
\right\rvert^2\right]
dt_{k}dt_{k-1}\cdots dt_{1}
\right)^{1/2}.
\end{split}\label{c06}\end{align}
Observe that
\eqref{c01}, 
\eqref{p01},
the fact that $\E\! \left[
\lVert W_1^0\rVert^4\right]<\infty$, and the fact that $f,\gamma\in C(\R^d,\R)$
 show for all $h\in \{f,\gamma\}$ that
\begin{align}
\left(
\E\! \left[
\left\lvert h(W_1^0)\right\rvert^4\right]\right)^{1/4}\leq
\left(
\E\! \left[
\left\lVert W_1^0\right\rVert^4\right]\right)^{1/4}
 \sup_{x\in\R^d\setminus\{0\}}
\tfrac{\lvert h(x)\rvert}{\lVert x\rVert}
\leq 
\left(
\E\! \left[
\left\lVert W_1^0\right\rVert^4\right]\right)^{1/4}
\sup_{x\in\R^d\colon \lVert x\rVert=1}
\lvert h(x)\rvert<\infty.\label{t18}
\end{align}%
This implies that $\max\{c_1,c_2\}<\infty$.
Next, \eqref{c01}, \eqref{p01}, the scaling properties of Brownian motions, and \eqref{t15}  imply for all $t\in[0,1)$, $s\in (t,1]$ that
\begin{align}\begin{split}
&
\left(
\E\!\left[
\left \lvert \gamma(W_s^0-W_t^0)\right\rvert^4\right]\right)^{1/4}\leq 
\sqrt{\lvert s-t\rvert}
\left(
\E\!\left[
\left \lvert \gamma\!\left(\tfrac{W_s^0-W_t^0}{\sqrt{\lvert s-t\rvert}}\right)\right\rvert^4\right]\right)^{1/4}
= c_2\sqrt{\lvert s-t\rvert}
\end{split}\label{c01c}\end{align}
and
\begin{align}\begin{split}
&
\left(
\E\!\left[
\left \lvert f(W_s^0-W_t^0)\right\rvert^4\right]\right)^{1/4}\leq 
\sqrt{\lvert s-t\rvert}
\left(
\E\!\left[
\left \lvert f\!\left(\tfrac{W_s^0-W_t^0}{\sqrt{\lvert s-t\rvert}}\right)\right\rvert^4\right]\right)^{1/4}
= c_1\sqrt{\lvert s-t\rvert}.
\end{split}\label{c01b}\end{align}
This, H\"older's inequality, the fact that Brownian motions have independent increments,   the fact that
$\forall\, t\in[0,1),s\in (t,1]\colon\varrho (t,s)(s-t)=  \frac{s-t}{2\sqrt{1-t}\sqrt{s-t}} =\frac{\sqrt{s-t}}{2\sqrt{1-t}}$, and
\eqref{t15}
prove that for all 
$\eta\in\Theta$,
$k\in\N_0$, 
$ (t_j)_{j\in[0,k]\cap\Z}\in [0,1)^{k+1} $
 with $\forall\, j\in[0,k)\colon t_j<t_{j+1}$ it holds that
\begin{align}\begin{split}
&
\left(\E\!\left[
\left \lvert \gamma\left( W^{\eta}_1-W^{\eta}_{t_k}\right)
 \tfrac{ f(
  W^{\eta}_{1}- W^{\eta}_{t_k})}{ 1 - t_k }
\prod_{j=1}^{k}
 \tfrac{f(W_{t_j}^0-W^0_{t_{j-1}})}{\sqrt{\varrho(t_{j-1},t_j)}(t_{j}-t_{j-1})}
\right \rvert^2\right]\right)^{1/2}
\\
&\leq
 \left(
\E\!\left[\left\lvert\gamma\left(
W_1^\eta-W_{t_k}^\eta \right)\right\rvert^4\right]\right)^{1/4}
\left(
\E\!\left[\left \lvert  
 \tfrac{f(W_1^\eta-W_{t_k}^\eta)}{1-t_k}
 \prod_{j=1}^{k}
 \tfrac{f(
W_{t_j}^0-W^0_{t_{j-1}})}{\sqrt{\varrho(t_{j-1},t_j)}(t_{j}-t_{j-1})
}
\right \rvert ^4\right]\right)^{1/4}\\
&=  \left(
\E\!\left[\left\lvert\gamma\left(
W_1^0-W_{t_k}^0 \right)\right\rvert^4\right]\right)^{1/4}
\left(
\E\!\left[ \left\lvert
 \tfrac{f(W_1^0-W_{t_k}^0)}{1-t_k}
\right\rvert^4\right]\right)^{1/4}
 \prod_{j=1}^{k}
 \tfrac{\left(
\E\!\left[\left \lvert  f(
W_{t_j}^0-W^0_{t_{j-1}})\right \rvert ^4\right]\right)^{1/4}}{\sqrt{\varrho(t_{j-1},t_j)}(t_{j}-t_{j-1})
}
\\
&\leq c_2 \sqrt{1-t_{k}} \tfrac{c_1\sqrt{1-t_k}}{1-t_k}
\prod_{j=1}^{k}
 \tfrac{c_1\sqrt{t_j-t_{j-1}}}{\sqrt{\varrho(t_{j-1},t_j)}(t_{j}-t_{j-1})
}= c_2 c_1^{k+1}
\prod_{j=1}^{k}
 \tfrac{1}{\sqrt{\varrho(t_{j-1},t_j)}\sqrt{t_j-t_{j-1}}}\\
&=  c_2 c_1^{k+1}2^{k/2}
\left[
\prod_{j=1}^{k}
 \tfrac{\sqrt{1-t_{j-1}}}{\sqrt{t_j-t_{j-1}}}\right]^{1/2}.
\end{split}\label{t25}
\end{align}
Next, 
 independence and distributional properties of MLP approximations (see \cite[Lemma~3.2]{hutzenthaler2021overcoming}), 
the fact that Brownian motions have independent increments,
the disintegration theorem (see, e.g., \cite[Lemma 2.2]{hutzenthaler2020overcoming}),
and
the fact that
$\forall\,t\in [0,1), h\in C([t,1],[0,\infty))\colon \E \!\left[h(\uniform_t^0)\right]=\int_{t}^{1}h(s)\varrho(t,s)\,ds  $
(cf. 
\cref{d14}) show that for all $n,m,i\in\N$,
$\ell\in [0,n-1]\cap\Z$,
$\iota\in\{0,1\}$, $\lambda\in \{-\ell,\ell\} $,
$k\in\N_0$, 
$\xi\in\R^d$,
$ (t_j)_{j\in[0,k]\cap\Z}\in [0,1)^{k+1} $
 with $\forall\, j\in[0,k)\colon t_j<t_{j+1}$ it holds that 
\begin{align}
&\E\!\left[\left\lvert
   \tfrac{
\bigl( f\circ
  {V}_{\ell-\iota ,m}^{(0,\lambda  ,i)}\bigr)
  \bigl(\mathcal{R}^{(0, \ell ,i)}_{t_k},\xi+W^0_{t_k}-W^0_{t_0}+W_{\mathcal{R}^{(0, \ell ,i)}_{t_k}}^{(0,\ell ,i)}-W_{t_k}^{(0,\ell ,i)}\bigr)
  }
  {\varrho(t_k,\mathcal{R}^{(0, \ell ,i)}_{t_k})}
   \tfrac{ f\bigl(
  W_{\mathcal{R}^{(0, \ell ,i)}_{t_k}}^{(0,\ell ,i)}- W^{(0, \ell , i)}_{t_k}\bigr)
  }{ \mathcal{R}^{(0, \ell ,i)}_{t_k}-t_k}
\prod_{j=1}^{k}
 \tfrac{f(
W_{t_j}^0-W^0_{t_{j-1}})}{\sqrt{\varrho(t_{j-1},t_j)}(t_{j}-t_{j-1})
}
\right\rvert^2
\right]\nonumber\\
&=\E\!\left[\left\lvert
   \tfrac{
\bigl( f\circ
  {V}_{\ell-\iota ,m}^{0}\bigr)
  \bigl(\mathcal{R}^{0}_{t_k},\xi+W^0_{t_k}-W^0_{t_0}+W_{\mathcal{R}^{0}_{t_k}}^{0}-W_{t_k}^{0}\bigr)
  }
  {\varrho(t_k,\mathcal{R}^{0}_{t_k})}
   \tfrac{ f\bigl(
  W_{\mathcal{R}^{0}_{t_k}}^{0}- W^{0}_{t_k}\bigr)
  }{ \mathcal{R}^{0}_{t_k}-t_k}
\prod_{j=1}^{k}
 \tfrac{f(
W_{t_j}^0-W^0_{t_{j-1}})}{\sqrt{\varrho(t_{j-1},t_j)}(t_{j}-t_{j-1})
}
\right\rvert^2
\right]\nonumber\\
&= 
\int_{t_k}^{1}\E\!\left[\left \lvert
   \tfrac{\bigl(f\circ  {V}_{\ell -\iota,m}^{0}\bigr)
  \bigl(t_{k+1},\xi+W^0_{t_k}-W^0_{t_0}+W_{t_{k+1}}^{0}-W_{t_k}^{0}\bigr)
  }
  {\varrho(t_k,t_{k+1})}
   \tfrac{ f\bigl(
  W_{t_{k+1}}^{0}- W^{0}_{t_k}\bigr)
  }{ t_{k+1}-t_k}
 \prod_{j=1}^{k}
 \tfrac{f\!\left(
W_{t_j}^{0}-W^{0}_{t_{j-1}}\right)}{\sqrt{\varrho(t_{j-1},t_j)}(t_{j}-t_{j-1})
}\right \rvert^2\right]\nonumber\\
&\qquad\qquad\qquad\qquad\qquad\qquad\qquad\qquad\qquad\qquad\qquad\qquad\qquad\qquad\qquad\qquad\qquad\qquad
\varrho(t_k,t_{k+1})\,
dt_{k+1}
\nonumber\\
&= 
\int_{t_k}^{1}\E\!\left[\left \lvert
(f\circ {V}_{\ell-\iota ,m}^{0})\!
  \left(t_{k+1},\xi+W_{t_{k+1}}^{0}-W^0_{t_0}\right)
 \prod_{j=1}^{k+1}
 \tfrac{f(
W_{t_j}^0-W^0_{t_{j-1}})}{\sqrt{\varrho(t_{j-1},t_j)}(t_{j}-t_{j-1})
}
\right \rvert^2\right]
dt_{k+1}
.\label{c04}
\end{align}
Next, 
\eqref{c10},
 \eqref{p01}, 
\eqref{c01},
 and the triangle inequality prove for all
$n,m\in\N$, $t\in [0,1) $, $x\in \R^d$ that $f(0)=0$ and
\begin{equation}  \begin{split}
  &\left\lvert(f\circ {V}_{n,m}^{0})(t,x) \right\rvert
  \leq  
  \sum_{i=1}^{m^n}\left\lvert f\left( \tfrac{g\bigl(x+W^{(0,0,-i)}_1-W^{(0,0,-i)}_t\bigr)-g(x)}{m^n}  
   \tfrac{ 
  W^{(0, 0, -i)}_{1}- W^{(0, 0, -i)}_{t}
  }{ 1 - t }\right)\right\rvert
  \\
  &\quad +\sum_{\ell=0}^{n-1}\sum_{i=1}^{m^{n-\ell }}\left\lvert f\left(
   \tfrac{
\bigl[\bigl(f\circ
  {V}_{\ell ,m}^{(0,\ell ,i)}\bigr)-\1_{\N}(\ell )\bigl(f\circ {V}_{\ell -1,m}^{(0,-\ell ,i)}\bigr)\bigr]
  \bigl(\mathcal{R}^{(0, \ell ,i)}_t,x+W_{\mathcal{R}^{(0, \ell ,i)}_t}^{(0,\ell ,i)}-W_t^{(0,\ell ,i)}\bigr)
  }
  {m^{n-\ell }\varrho(t,\mathcal{R}^{(0, \ell ,i)}_t)}
   \tfrac{ 
  W_{\mathcal{R}^{(0, \ell ,i)}_t}^{(0,\ell ,i)}- W^{(0, \ell , i)}_{t}
  }{ \mathcal{R}^{(0, \ell ,i)}_t-t}\right)\right\rvert\\
&\leq 
  \sum_{i=1}^{m^n} \tfrac{ \gamma\left( W^{(0,0,-i)}_1-W^{(0,0,-i)}_t\right)}{m^n}
\left \lvert  
   \tfrac{ f\!\left(
  W^{(0, 0, -i)}_{1}- W^{(0, 0, -i)}_{t}\right)
  }{ 1 - t }\right\rvert
  \\
  &\quad +\sum_{\ell=1}^{n-1}\sum_{i=1}^{m^{n-\ell }}
   \tfrac{\left\lvert
\left[ f\circ
  {V}_{\ell ,m}^{(0,\ell ,i)}- f\circ {V}_{\ell -1,m}^{(0,-\ell ,i)}\right]\!
  \bigl(\mathcal{R}^{(0, \ell ,i)}_t,x+W_{\mathcal{R}^{(0, \ell ,i)}_t}^{(0,\ell ,i)}-W_t^{(0,\ell ,i)}\bigr)\right\rvert
  }
  {m^{n-\ell }\varrho(t,\mathcal{R}^{(0, \ell ,i)}_t)}
   \tfrac{ \left\lvert f\bigl(
  W_{\mathcal{R}^{(0, \ell ,i)}_t}^{(0,\ell ,i)}- W^{(0, \ell , i)}_{t}\bigr)
  \right\rvert}{ \mathcal{R}^{(0, \ell ,i)}_t-t}.
\end{split} \label{c04b} \end{equation}
This, the triangle inequality, \eqref{t25}, and \eqref{c04}
show that for all
$m,n\in\N$, 
$k\in\N_0$,
$\xi\in\R^d$,
$ (t_j)_{j\in[0,k]\cap\Z}\in [0,1)^{k+1} $
 with $\forall\, j\in[0,k)\colon t_j<t_{j+1}$ it holds that
\begin{align}
&
\left(\E\!\left[
\left\lvert
\left(f
\circ 
V_{n,m}^0\right)(t_k,\xi + W^0_{t_k}-W_{t_0}^0)
\prod_{j=1}^{k}
\tfrac{f(W_{t_j}^0-W^0_{t_{j-1}})}{\sqrt{\varrho(t_{j-1},t_j)}(t_{j}-t_{j-1})}
\right\rvert^2\right]\right)^{1/2}\nonumber
\\
&\leq  \sum_{i=1}^{m^n}\left(\E\!\left[
\left\lvert
 \tfrac{ \gamma\left( W^{(0,0,-i)}_1-W^{(0,0,-i)}_{t_k}\right)}{m^n}
   \tfrac{ f\!\left(
  W^{(0, 0, -i)}_{1}- W^{(0, 0, -i)}_{t_k}\right)
  }{ 1 - {t_k} }
\prod_{j=1}^{k}
\tfrac{f(W_{t_j}^0-W^0_{t_{j-1}})}{\sqrt{\varrho(t_{j-1},t_j)}(t_{j}-t_{j-1})}
\right\rvert^2\right]\right)^{1/2}\nonumber
\\
&+\sum_{\ell=1}^{n-1}\sum_{i=1}^{m^{n-\ell }}\left(\E\!\left[
\left\lvert
   \tfrac{
\left[ f\circ
  {V}_{\ell ,m}^{(0,\ell ,i)}\right]\!
  \bigl(\mathcal{R}^{(0, \ell ,i)}_{t_k},\xi+W_{\mathcal{R}^{(0, \ell ,i)}_{t_k}}^{(0,\ell ,i)}-W_{t_k}^{(0,\ell ,i)}\bigr)
  }
  {m^{n-\ell }\varrho({t_k},\mathcal{R}^{(0, \ell ,i)}_{t_k})}
   \tfrac{ f\bigl(
  W_{\mathcal{R}^{(0, \ell ,i)}_{t_k}}^{(0,\ell ,i)}- W^{(0, \ell , i)}_{t_k}\bigr)
  }{ \mathcal{R}^{(0, \ell ,i)}_{t_k} - {t_k}}
\prod_{j=1}^{k}
\tfrac{f(W_{t_j}^0-W^0_{t_{j-1}})}{\sqrt{\varrho(t_{j-1},t_j)}(t_{j}-t_{j-1})}
\right\rvert^2\right]\right)^{1/2}\nonumber
\\
&+\sum_{\ell=1}^{n-1}\sum_{i=1}^{m^{n-\ell }}\left(\E\!\left[
\left\lvert
   \tfrac{
\left[f\circ {V}_{\ell -1,m}^{(0,-\ell ,i)}\right]\!
  \bigl(\mathcal{R}^{(0, \ell ,i)}_{t_k},\xi +W_{\mathcal{R}^{(0, \ell ,i)}_{t_k}}^{(0,\ell ,i)}-W_{t_k}^{(0,\ell ,i)}\bigr)
  }
  {m^{n-\ell }\varrho({t_k},\mathcal{R}^{(0, \ell ,i)}_{t_k})}
   \tfrac{ f\bigl(
  W_{\mathcal{R}^{(0, \ell ,i)}_{t_k}}^{(0,\ell ,i)}- W^{(0, \ell , i)}_{t_k}\bigr)
  }{ \mathcal{R}^{(0, \ell ,i)}_{t_k}-{t_k}}
\prod_{j=1}^{k}
\tfrac{f(W_{t_j}^0-W^0_{t_{j-1}})}{\sqrt{\varrho(t_{j-1},t_j)}(t_{j}-t_{j-1})}
\right\rvert^2\right]\right)^{1/2}\nonumber\\
&\leq c_2 c_1^{k+1}2^{k/2}
\left[
\prod_{j=1}^{k}
 \tfrac{\sqrt{1-t_{j-1}}}{\sqrt{t_j-t_{j-1}}}\right]^{1/2}\nonumber\\
&\qquad
+2\sum_{\ell=1}^{n-1}\left(\int_{t_k}^{1}\E\!\left[\left \lvert
(f\circ {V}_{\ell ,m}^{0})\!
  \left(t_{k+1},\xi+W_{t_{k+1}}^{0}-W^0_{t_0}\right)
 \prod_{j=1}^{k+1}
 \tfrac{f(
W_{t_j}^0-W^0_{t_{j-1}})}{\sqrt{\varrho(t_{j-1},t_j)}(t_{j}-t_{j-1})
}
\right \rvert^2\right]
dt_{k+1}\right)^{1/2}.\label{c20}
\end{align}
Moreover, for all $k,N,n\in \N$,   $\ell\in[1,n-1]\cap\Z$ with $k+n\leq  N$ it holds that $(k+1)+\ell\leq N$ and $c_2 c_1^{k+1}2^{k}\leq 
c_2 \max\{c_1^{N},1\}2^{N-1}$.
This, 
\eqref{c20},
the triangle inequality, 
\cref{c02}, and \eqref{c06}
show that for all
$N,m,n,k\in\N$,
$t_0\in[0,1)$, $\xi\in\R^d$
 with $k+n\leq N $
it holds that
\begin{align}
&
\left(\int_{t_0}^1\int_{t_1}^{1}\ldots\int_{t_{k-1}}^{1}
\E\!\left[
\left\lvert
\left(f
\circ 
V_{n,m}^0\right)(t_k,\xi + W^0_{t_k}-W_{t_0}^0)
\prod_{j=1}^{k}
\tfrac{f(W_{t_j}^0-W^0_{t_{j-1}})}{\sqrt{\varrho(t_{j-1},t_j)}(t_{j}-t_{j-1})}
\right\rvert^2\right]dt_{k}dt_{k-1}\cdots dt_{1}\right)^{1/2}\nonumber\\
&\leq c_2 c_1^{k+1}2^{k/2}\left(\int_{t_0}^1\int_{t_1}^{1}\ldots\int_{t_{k-1}}^{1}
\prod_{j=1}^{k}
 \tfrac{\sqrt{1-t_{j-1}}}{\sqrt{t_j-t_{j-1}}}
dt_{k}dt_{k-1}\cdots dt_{1}\right)^{1/2}\nonumber\\
&+2\sum_{\ell=1}^{n-1}
\vastleft{25pt}(\int_{t_0}^1\int_{t_1}^{1}\ldots\int_{t_k}^{1}
\E\!\vastleft{23pt}[\vastleft{23pt} \lvert
(f\circ {V}_{\ell ,m}^{0})\!
  \left(t_{k+1},\xi+W_{t_{k+1}}^{0}-W_{t_0}^0\right)
\nonumber\\
&\qquad\qquad
\qquad\qquad\qquad\qquad\qquad\qquad\qquad\qquad
 \prod_{j=1}^{k+1}
 \tfrac{f(
W_{t_j}^0-W^0_{t_{j-1}})}{\sqrt{\varrho(t_{j-1},t_j)}(t_{j}-t_{j-1})
}
\vastright{23pt} \rvert^2\vastright{23pt}]
dt_{k+1}dt_k\cdots dt_1
\vastright{25pt})^{1/2}\nonumber\\
&\leq c_2 c_1^{k+1}2^{k/2}2^{k/2}+2\sum_{\ell=1}^{n-1}\epsUp_{N,m}(\ell)\leq c_2 \max\{c_1^{N},1\}2^{N-1}+2\sum_{\ell = 1}^{n-1}\epsUp_{N,m}(\ell).
\end{align}
This and \eqref{c06}
show 
for all
$N,m\in \N$, $n\in[1,N-1]\cap\N$ that
\begin{align}
&
\epsUp_{N,m}(n)
\leq c_2 \max\{c_1^{N},1\}2^{N-1}+2\sum_{\ell = 1}^{n-1}\epsUp_{N,m}(\ell).
\label{c05}\end{align}
We will prove  for all $m,N\in\N$, $n\in[1,N-1]\cap\N$ that
$\epsUp_{N,m}(n)\leq c_2 \max\{c_1^{N},1\}2^{N-1}3^{n-1}$
by induction on $n$. 
First, \eqref{c05} shows 
for all $N,m\in\N $ that
$
\epsUp_{N,m}(1)\leq 
c_2 \max\{c_1^{N},1\}2^{N-1}
$. This establishes
the base case $n=1 $. For the induction step let $m,N\in\N$, $n\in[1,N-1]\cap\N$ satisfy that for all
$\ell\in [1,n-1]\cap\N$ it holds that $\epsUp_{N,m}(\ell)\leq c_2 \max\{c_1^{N},1\}2^{N-1}3^{\ell-1}$. Then \eqref{c05} 
shows that
\begin{align}\begin{split}
&
\epsUp_{N,m}(n)\leq
c_2 \max\{c_1^{N},1\}2^{N-1}\left[1+
2\sum_{\ell = 1}^{n-1}3^{\ell-1}\right]
=
c_2 \max\{c_1^{N},1\}2^{N-1}3^{n-1}.
\end{split}\end{align}
This shows the induction step. Induction hence proves 
for all $m,N\in\N$, $n\in[1,N-1]\cap\N$ that
\begin{align}\label{c15}
\epsUp_{N,m}(n)\leq c_2 \max\{c_1^{N},1\}2^{N-1}3^{n-1}.
\end{align}
This, \eqref{c20},  \eqref{c06},
 \eqref{t15}, and \eqref{t18}
 show for all $n,m\in\N$, $t_0\in[0,1)$, $\xi\in\R^d$ that
\begin{align}\begin{split}
&
\left(\E\!\left[
\left\lvert
\left(f
\circ 
V_{n,m}^0\right)(t_0, \xi)
\right\rvert^2\right]
\right)^{1/2} \\
&\leq 
c_2 c_1
+2\sum_{\ell =1}^{n-1}
\Biggl(\int_{t_0}^1
\E\Biggl[\biggl \lvert
(f\circ {V}_{\ell ,m}^{0})\!
  \left(t_{1},\xi+W_{t_{1}}^{0}-W_{t_0}^0\right)
 \tfrac{f(
W_{t_1}^0-W^0_{t_{0}})}{\sqrt{\varrho(t_{0},t_1)}(t_{1}-t_{0})
}
\biggr \rvert^2\Biggr]
dt_1
\Biggr)^{1/2} \\
&\leq c_1c_2+ 2\sum_{\ell=1}^{n-1}\epsUp_{\ell+1,m}(\ell)
\leq c_1c_2+ 2\sum_{\ell=1}^{n-1}\left[c_2\max\{c_1^{\ell+1},1\}2^{\ell}3^{\ell-1}\right]
\\
&\leq  c_2\max\{c_1^{n},1\}2^{n-1}\left[1+2\sum_{\ell=1}^{n-1}3^{\ell-1}\right] 
=c_2\max\{c_1^{n},1\}2^{n-1}\left[
1+2
\tfrac{3^{n-1}-1}{2}\right]\\
&
=
c_2\max\{c_1^{n},1\}2^{n-1}3^{n-1} =
\left(\E\!\left[\lvert \gamma( W^0_1) \rvert^4\right]\right)^{1/4}
 \max\left\{\left(\E\!\left[\lvert f( W^0_1)\rvert^4\right]\right)^{n/4},1\right\}6^{n-1}<\infty
.
\label{c16}\end{split}
\end{align}
This proves \eqref{c13}.

Next, \eqref{c01}, Jensen's inequality, and \eqref{t15} show for all
$\theta\in \Theta$ that
\begin{align}
&
 \left(\E\!\left[
\left\lvert g(W_1^\theta)-g(0)\right\rvert^2 \right]\right)^{1/2}
\leq \left(\E\!\left[
\left\lvert \gamma(W_1^0)\right\rvert^2 \right]\right)^{1/2}\leq 
\left(\E\!\left[
\left\lvert \gamma(W_1^0)\right\rvert^4 \right]\right)^{1/4}
=c_2.\label{t16b}
\end{align}
Moreover, \eqref{c01}, H\"older's inequality,
\eqref{t15},  and
 the fact that $\lVert W^0_1\rVert^2$ is chi-square distributed with $d$ degrees of freedom show 
 for all
$\theta\in \Theta$
that
\begin{align}\begin{split}
&
 \left(\E\!\left[
\left\lvert g(W_1^\theta)-g(0)\right\rvert^2 \left\lVert W_1^\theta\right\rVert^2\right]\right)^{1/2}
\leq \left(\E\!\left[
\left\lvert \gamma(W_1^\theta)\right\rvert^2 \left\lVert W_1^\theta\right\rVert^2\right]\right)^{1/2}\\
&\leq 
\left(\E\!\left[
\left\lvert \gamma(W_1^0)\right\rvert^4 \right]\right)^{1/4}
\left(\E\!\left[ \left\lVert W_1^0\right\rVert^4\right]\right)^{1/4}
=c_2(d(d+2))^{1/4}\leq c_2\sqrt{d+1}.\label{t16}
\end{split}\end{align}
Furthermore,  independence and distributional properties of MLP approximations (see \cite[Lemma~3.2]{hutzenthaler2021overcoming}),
the disintegration theorem (see, e.g., \cite[Lemma 2.2]{hutzenthaler2020overcoming}),
the fact that $\forall\,b\in [0,1] \colon
\P(\uniform^0_0\leq b)=
\P(\unif^0\leq b)=\int_{0}^{b}\varrho (0,s)\,ds$, 
 the fact that
$\E[\lVert W_1^0 \rVert^2]=d$, 
\eqref{c16},  the fact that
$\int_{0}^{1}2\sqrt{s}\,ds=\frac{2s^{1.5}}{1.5}\bigr|_{s=0}^1= \frac{4}{3}$,
and the fact that
$\int_{0}^{1}\frac{2\,ds}{\sqrt{s}} = 4$
 show for all $\ell,m,i\in\N$,
$\iota\in\{0,1\}$,
$\lambda\in \{-\ell,\ell\} $
 that
\begin{align}
&
\left(
\E\!\left[\left\lvert
  \tfrac{(f\circ V_{\ell-\iota,m}^{(0,\lambda,i)})(\uniform_0^{(0,\ell,i)},W^{(0,\ell,i)}_{\uniform_0^{(0,\ell,i)}})}{\varrho(0,\uniform_0^{(0,\ell,i)})}  
\right\rvert^2\right]\right)^{1/2}=
\left(
\E\!\left[\left\lvert
  \tfrac{(f\circ V_{\ell-\iota,m}^0)(\uniform_0^0,W^0_{\uniform_0^0})}{\varrho(0,\uniform_0^0)}  
\right\rvert^2\right]\right)^{1/2}\nonumber\\
&
=
\left(\int_{0}^{1}
\E\!\left[\left\lvert
  \tfrac{(f\circ V_{\ell-\iota,m}^0)(s,W^0_{s})}{\varrho(0,s)}  \right\rvert^2\right]
\varrho(0,s)\,ds
\right)^{1/2}=
\left(\int_{0}^{1}  \tfrac{\E\!\left[\E\left[\lvert(f\circ V_{\ell-\iota,m}^0)(s,x)\rvert^2
\right]\bigr|_{x=W_s^0}\right]}{\varrho(0,s)}\,ds
\right)^{1/2}\nonumber\\
&\leq 
c_2 \max\{c_1^{\ell-\iota},1\}6^{\ell-\iota-1}
 \left(\int_{0}^{1}  \tfrac{ds}{ \tfrac{1}{2\sqrt{s}}}\right)^{1/2}
\leq 
c_2 \max\{c_1^{\ell},1\}6^{\ell-1}
 \left(\int_{0}^{1}  2\sqrt{s}\,ds\right)^{1/2}\nonumber
\\
&=  c_2 \max\{c_1^{\ell},1\}6^{\ell-1} \tfrac{2}{\sqrt{3}} \leq c_2 \max\{c_1^{\ell},1\}6^\ell\label{c11b}
\end{align}
and
\begin{align}
&
\left(
\E\!\left[\left\lVert
 \tfrac{(f\circ V_{\ell-\iota,m}^{(0,\lambda,i)})(\uniform_0^{(0,\ell,i)},W^{(0,\ell,i)}_{\uniform_0^{(0,\ell,i)}})}{\varrho(0,\uniform_0^{(0,\ell,i)})}  \frac{W^{(0,\ell,i)}_{\uniform_0^{(0,\ell,i)}}}{\uniform_0^{(0,\ell,i)}}\right\rVert^2\right]\right)^{1/2}
= 
\left(
\E\!\left[\left\lVert
  \tfrac{(f\circ V_{\ell-\iota,m}^0)(\uniform_0^0,W^0_{\uniform_0^0})}{\varrho(0,\uniform_0^0)}  \tfrac{W^0_{\uniform_0^0}}{\uniform_0^0}\right\rVert^2\right]\right)^{1/2}\nonumber\\
&
=
\left(\int_{0}^{1}
\E\!\left[\left\lVert
  \tfrac{(f\circ V_{\ell-\iota,m}^0)(s,W^0_{s})}{\varrho(0,s)}  \tfrac{W^0_{s}}{s}\right\rVert^2\right]
\varrho(0,s)\,ds
\right)^{1/2}
=\left(\int_{0}^{1}  \tfrac{\E\!\left[\left(\E\left[\lvert(f\circ V_{\ell-\iota,m}^0)(s,x)\rvert^2\right] \lVert x\rVert^2
\right)
\Bigr|_{x=W_s^0}
\right]  
}{s^2\varrho(0,s)}\,ds
\right)^{1/2}\nonumber\\
&\leq 
c_2 \max\{c_1^{\ell},1\}6^{\ell-1}
 \left(\int_{0}^{1}  \tfrac{\E\bigl[\lVert W^0_s\rVert^2\bigr]\,ds}{s^2  \frac{1}{2\sqrt{s}}}\right)^{1/2}
=
c_2 \max\{c_1^{\ell},1\}6^{\ell-1}\sqrt{d}
 \left(\int_{0}^{1}  \frac{2\,ds}{\sqrt{s}}\right)^{1/2}\nonumber
\\
&= c_2 \max\{c_1^{\ell},1\}6^{\ell-1}2\sqrt{d}
\leq c_2 \max\{c_1^{\ell},1\}6^\ell\sqrt{d}
.
\label{c11}
\end{align}
Next, \eqref{c10} 
shows for all 
$n,m\in\N$ that
\begin{equation}  \begin{split}
  &{U}_{n,m}^{0}(0,0) -g(0)
 =   
  \sum_{i=1}^{m^n} \tfrac{g\bigl(W^{(0,0,-i)}_1\bigr)-g(0)}{m^n}  
  +\sum_{\ell =1}^{n-1}\sum_{i=1}^{m^{n-\ell }}
   \tfrac{
\left( f\circ
  {V}_{\ell ,m}^{(0,\ell ,i)}-  f\circ {V}_{\ell -1,m}^{(0,-\ell ,i)}\right)\!
  \bigl(\mathcal{R}^{(0, \ell ,i)}_0,W_{\mathcal{R}^{(0, \ell ,i)}_0}^{(0,\ell ,i)}\bigr)
  }
  {m^{n-\ell }\varrho(0,\mathcal{R}^{(0, \ell ,i)}_0)}.
\end{split}     \end{equation}
This,
 the triangle inequality, \eqref{t16b}, and \eqref{c11b}
show for all $n,m\in \N$ that
\begin{align}
&\left(
\E\!\left[
\left\lvert{U}_{n,m}^{0}(0,0)-g(0)\right \rvert^2\right]\right)^{1/2}
\leq \sum_{i=1}^{m^n}
\left(
\E\!\left[
\left\lvert
 \tfrac{g\bigl(W^{(0,0,-i)}_1\bigr)-g(0)}{m^n}  
\right \rvert^2\right]\right)^{1/2}\nonumber\\
&+
\sum_{\ell =1}^{n-1}\sum_{i=1}^{m^{n-\ell }}\left[
\left(
\E\!\left[
\left\lvert
   \tfrac{
\left( f\circ
  {V}_{\ell ,m}^{(0,\ell ,i)}\right)\!
  \bigl(\mathcal{R}^{(0, \ell ,i)}_0,W_{\mathcal{R}^{(0, \ell ,i)}_0}^{(0,\ell ,i)}\bigr)
  }
  {m^{n-\ell }\varrho(0,\mathcal{R}^{(0, \ell ,i)}_0)}
\right \rvert^2\right]\right)^{1/2}
+
\left(
\E\!\left[
\left\lvert
   \tfrac{
\left( f\circ
  {V}_{\ell -1,m}^{(0,-\ell ,i)}\right)\!
  \bigl(\mathcal{R}^{(0, \ell ,i)}_0,W_{\mathcal{R}^{(0, \ell ,i)}_0}^{(0,\ell ,i)}\bigr)
  }
  {m^{n-\ell }\varrho(0,\mathcal{R}^{(0, \ell ,i)}_0)}
\right \rvert^2\right]\right)^{1/2}\right]\nonumber\\
&\leq c_2+
2\sum_{\ell=1}^{n-1} \left[c_2 \max\{c_1^{\ell},1\}6^\ell\right]
\leq 
2\sum_{\ell=0}^{n-1} \left[c_2 \max\{c_1^{\ell},1\}6^\ell\right]
\leq 2c_2 \max\{c_1^{n},1\}\sum_{\ell=0}^{n-1}6^\ell\nonumber\\
&=
2c_2 \max\{c_1^{n},1\}
\tfrac{6^n-1}{5}
 \leq 0.5\cdot c_2 \max\{c_1^{n},1\}6^n.
\label{d10}
\end{align}
Next, \eqref{c10} 
shows for all 
$n,m\in\N$ that
\begin{equation}  \begin{split}
&  {V}_{n,m}^{0}(0,0)\\
&=
  \sum_{i=1}^{m^n} \left[\tfrac{g\bigl(W^{(0,0,-i)}_1\bigr)-g(0)}{m^n}  
  W^{(0, 0, -i)}_{1}\right]
   +\sum_{\ell =1}^{n-1}\sum_{i=1}^{m^{n-\ell }}
\left[
   \tfrac{
\left( f\circ
  {V}_{\ell ,m}^{(0,\ell ,i)}- f\circ {V}_{\ell -1,m}^{(0,-\ell ,i)}\right)\!
  \bigl(\mathcal{R}^{(0, \ell ,i)}_0,W_{\mathcal{R}^{(0, \ell ,i)}_0}^{(0,\ell ,i)}\bigr)
  }
  {m^{n-\ell }\varrho(0,\mathcal{R}^{(0, \ell ,i)}_0)}
   \tfrac{ 
  W_{\mathcal{R}^{(0, \ell ,i)}_0}^{(0,\ell ,i)}
  }{ \mathcal{R}^{(0, \ell ,i)}_0}\right].
\end{split}  \label{c11c}   \end{equation}
This, the triangle inequality, \eqref{t16}, \eqref{c11}, and the fact
$\sqrt{d+1}\leq 2\sqrt{d}$
 show for all
$n,m\in\N $ that
\begin{align}\begin{split}
&
\left(
\E\!\left[
\left\lVert
{V}_{n,m}^{0}(0,0)\right \rVert^2\right]\right)^{1/2}
\leq   \sum_{i=1}^{m^n} 
\left(
\E\!\left[
\left\lVert
\tfrac{g\bigl(W^{(0,0,-i)}_1\bigr)-g(0)}{m^n}  
  W^{(0, 0, -i)}_{1}
\right \rVert^2\right]\right)^{1/2}
 \\
& +\sum_{\ell =1}^{n-1}\sum_{i=1}^{m^{n-\ell }}
\left(
\E\!\left[
\left\lVert
   \tfrac{
\left( f\circ
  {V}_{\ell ,m}^{(0,\ell ,i)}\right)\!
  \bigl(\mathcal{R}^{(0, \ell ,i)}_0,W_{\mathcal{R}^{(0, \ell ,i)}_0}^{(0,\ell ,i)}\bigr)
  }
  {m^{n-\ell }\varrho(0,\mathcal{R}^{(0, \ell ,i)}_0)}
   \tfrac{ 
  W_{\mathcal{R}^{(0, \ell ,i)}_0}^{(0,\ell ,i)}
  }{ \mathcal{R}^{(0, \ell ,i)}_0}
\right \rVert^2\right]\right)^{1/2}
\\
&
+\sum_{\ell =1}^{n-1}\sum_{i=1}^{m^{n-\ell }}\left(
\E\!\left[
\left\lVert
   \tfrac{
\left(  f\circ {V}_{\ell -1,m}^{(0,-\ell ,i)}\right)\!
  \bigl(\mathcal{R}^{(0, \ell ,i)}_0,W_{\mathcal{R}^{(0, \ell ,i)}_0}^{(0,\ell ,i)}\bigr)
  }
  {m^{n-\ell }\varrho(0,\mathcal{R}^{(0, \ell ,i)}_0)}
   \tfrac{ 
  W_{\mathcal{R}^{(0, \ell ,i)}_0}^{(0,\ell ,i)}
  }{ \mathcal{R}^{(0, \ell ,i)}_0}\right \rVert^2\right]\right)^{1/2}
\\
&\leq c_2\sqrt{d+1}+
2\sum_{\ell=1}^{n-1} \left[c_2 \max\{c_1^{\ell},1\}6^\ell\sqrt{d}\right]
\leq 
2\sum_{\ell=0}^{n-1} \left[c_2 \max\{c_1^{\ell},1\}6^\ell\sqrt{d}\right]\\
&
\leq 2 c_2 \max\{c_1^{n},1\}\sqrt{d}\sum_{\ell=0}^{n-1}6^\ell
=2 c_2 \max\{c_1^{n},1\}\sqrt{d}\tfrac{6^n-1}{5}\leq 0.5\max\{c_1^{n},1\}6^n
\sqrt{d}
.
\end{split}\end{align}
This, the triangle inequality, \eqref{d10},
the fact that $0.5 \sqrt{d+1}\leq \sqrt{d}$, \eqref{t15}, and \eqref{t18}
 show for all $n,m\in\N$ that
\begin{align}
&\left(
\E\!\left[
\left\lvert{U}_{n,m}^{0}(0,0)\right \rvert^2\right]+
\E\!\left[
\left\lVert{V}_{n,m}^{0}(0,0)\right \rVert^2\right]\right)^{1/2}\leq \lvert g(0)\rvert+
\left(
\E\!\left[
\left\lvert{U}_{n,m}^{0}(0,0)-g(0)\right \rvert^2\right]+
\E\!\left[
\left\lVert{V}_{n,m}^{0}(0,0)\right \rVert^2\right]\right)^{1/2}\nonumber\\
&\leq \lvert g(0)\rvert+
0.5\cdot c_2 \max\{c_1^{n},1\}6^n\sqrt{1+d}
\leq \lvert g(0)\rvert + c_2 \max\{c_1^{n},1\}6^n\sqrt{d}\nonumber\\
&
= \lvert g(0)\rvert+
\left(\E\!\left[\lvert \gamma( W^0_1) \rvert^4\right]\right)^{1/4}
 \max\left\{\left(\E\!\left[\lvert f( W^0_1)\rvert^4\right]\right)^{n/4},1\right\}6^n\sqrt{d}<\infty
.
\end{align}
This 
 shows \eqref{c14}. This completes the proof of \cref{m01}.
\end{proof}

\section{Lower bounds for MLP approximations for our counterexample}
A difficulty in the derivation of lower bounds for MLP approximations
are the sums on the right-hand side of the MLP recursion \eqref{c10}.
Using symmetry we show that the summand on the right-hand side of \eqref{c10}
are centered in our case.
Thus second moments are equal to variances.
The following lemma then estimates the variance of the sum over independent random fields
evaluated at the same random and independent vector from below by the sum of the variances. 

\begin{lemma}\label{t01}Let $(\Omega, \mathcal{F},\P)$ be a  probability space,
let
$m,n\in\N$, 
let
$F_1,F_2,\ldots,F_n\colon \R^m\times \Omega\to\R $, be 
i.i.d.\ random fields, let 
$X\colon\Omega\to \R^m $ be a
random variable, 
assume that $\{F_1,F_2,\ldots,F_n\} $ and $X$ are independent,
and
assume 
that
$\E[\lvert F_1(X)\rvert^2]<\infty$.
Then 
$
\var \bigl[\sum_{i=1}^{n} F_i(X)\bigr] \geq n
\var \!\left( F_1(X)\right).$
\end{lemma}
\begin{proof}[Proof of \cref{t01}]
Throughout the proof let $A\subseteq \R^m$ satisfy that
$A=\{x\in\R^m\colon \E [\lvert F_1(x)\rvert^2]<\infty\}$.
Then the disintegration theorem (see, e.g., \cite[Lemma 2.2]{hutzenthaler2020overcoming}) 
and the fact that
$\E[\lvert F_1(X)\rvert^2]<\infty$
 show that
$
\infty\cdot \P(X\notin A)\leq \E \Bigl[\E [\lvert F_1(x)\rvert^2]\bigr|_{x=X}\1_{X\notin A}\Bigr]\leq \E \Bigl[\E [\lvert F_1(x)\rvert^2]\bigr|_{x=X}\Bigr]= \E \bigl[\lvert F_1(X)\rvert^2\bigr]<\infty.
$
Hence, $\P(X\notin A)=0$ and $\P(X\in A)=1$.
This, the disintegration theorem (see, e.g., \cite[Lemma 2.2]{hutzenthaler2020overcoming}),  the assumptions, and Jensen's inequality show that for all
$i,j\in \{1,2,\ldots,n\}$ with $i\neq j$ it holds that
\begin{align}\begin{split}
&
\var [F_i(X)]
=\E \!\left[\1_{X\in A}\lvert F_i(X)\rvert^2\right]
-\left\lvert\E [\1_{X\in A} F_i(X)]\right\rvert^2\\
&
=\E \!\left[\1_{X\in A}
\E\!\left[\lvert  F_i(x)\rvert^2 \right]\Bigr|_{x=X}\right]
-\left\lvert \E \!\left[ \1_{X\in A}\E\! \left[  F_i(x)\right]\Bigr|_{x=X}\right] \right\rvert^2
= \var[F_1(X)] \\
\end{split}\end{align}
and
\begin{align}\begin{split}
&
\E\!\left[F_i(X)F_j(X) \right]= \E\!\left[\1_{X\in A}\E\!\left[
 F_i(x)F_j(x)\right]\Bigr|_{x=X}\right]
=
\E\!\left[\1_{X\in A}
\E\!\left[
F_i(x)\right]\E\!\left[
F_j(x)\right]\Bigr|_{x=X}\right]\\
&
=
\E\!\left[\1_{X\in A} \Bigl\lvert\E\!\left[F_1(x)\right]\Bigr \rvert^2\Bigr|_{x=X}\right]\geq 
\left\lvert\E\!\left[\1_{X\in A} \E\!\left[F_1(x)\right]\Bigr|_{x=X}\right]\right\rvert^2\\
&
=
\E\!\left[ \1_{X\in A}\E\!\left[F_i(x)\right]\Bigr|_{x=X}\right]
\E\!\left[\1_{X\in A} \E\!\left[F_j(x)\right]\Bigr|_{x=X}\right]= \E\!\left[ F_i(X)\right]
\E\!\left[ F_j(X)\right].
\end{split}\end{align}
Therefore, it holds that
\begin{align}\small\begin{split}
\var \!\left(\sum_{i=1}^{n} F_i(X)\right) &= 
\sum_{i=1}^{n}
\var \!\left( F_i(X)\right) + \sum_{\substack{i,j=1\\i\neq j}}^{n}
\Bigl(
\E\!\left[F_i(X)F_j(X)\right]-\E\!\left[F_i(X)\right]\E\!\left[F_j(X)\right]
\Bigr)\geq n
\var \!\left( F_1(X)\right).
\end{split}\end{align}The proof of \cref{t01} is thus completed.
\end{proof}
It turns out that it is easier in case of our counterexample
to derive a recursion for the second moment of
the nonlinearity applied to the MLP approximations.
From this we then derive in \cref{c12} below a lower bound
for the approximation error of MLP approximations in the case of our counterexample.
\begin{theorem}\label{d01}
Assume \cref{s01} and assume
for all $v=(v_1,v_2,\ldots,v_d)\in\R^d$ that
\begin{align}
g(v)=\lvert v_1\rvert\quad\text{and}\quad f(v)=\left(\textstyle\sum_{j=2}^{d}\lvert v_j\rvert^2\right)^{1/2}.\label{t19}
\end{align}
Then for all 
$n,m\in\N$ with
$d\geq (1224m)^n n!$ it holds that
$\left(
\E\!\left[
\left\lvert
\left(f
\circ 
V_{n,m}^0\right)(0, 0)
\right\rvert^2
\right]
\right)^{1/2}\geq\tfrac{d^{n/2}}{(34m)^{n/2}  \sqrt{n!}}.
$
\end{theorem}
\begin{proof}[Proof of \cref{d01}]
Throughout this proof
for every $s\in [0,1]$, $i\in\{1,2,\ldots,d\}$ let $W_{i,s}^0$ be the $i$-th coordinate of $W_s^0$,
 let $\epsDown_{k,m}\colon\N_0\to[0,\infty]$, $k\in\N_0$,
$m\in\N $, satisfy for all $k,m\in \N $, 
$n\in\N_0$,
$t_0\in\{0\}$ that 
\begin{align}\label{c06d}
\epsDown_{0,m}(n)= \left(\E\!\left[
\left\lvert
\left(f
\circ 
V_{n,m}^0\right)(0, 0)
\right\rvert^2\right]\right)^{1/2}
\end{align}
and
\begin{align}
&
\epsDown_{k,m}(n)=
\left(\int_{t_0}^1\int_{t_1}^{1}\ldots\int_{t_{k-1}}^{1}
\E\!\left[
\left\lvert
\left(f
\circ 
V_{n,m}^0\right)(t_k, W^0_{t_k})
\prod_{j=1}^{k}
\tfrac{f(W_{t_j}^0-W^0_{t_{j-1}})}{\sqrt{\varrho(t_{j-1},t_j)}(t_{j}-t_{j-1})}
\right\rvert^2\right]
dt_{k}dt_{k-1}\cdots dt_{1}
\right)^{1/2},
\label{c06c}\end{align}
let 
$\widehat{W}^\theta\colon [0,1]\times\Omega\to\R^d$, $\theta\in\Theta$, satisfy for all 
$\theta\in\Theta$ that $\widehat{W}^\theta=-W^\theta$, and 
let $\widehat{V}^\theta_{n,m}\colon [0,1)\times\R^d\times\Omega\to \R^d$,
$n,m\in\Z$, $\theta\in\Theta$, satisfy 
for all $  n,m \in \N$, $ \theta \in \Theta $,
$ t\in [0,1)$,
$x \in \R^d$
that $
{\widehat{V}}_{-1,m}^{\theta}(t,x)=\widehat{V}_{0,m}^{\theta}(t,x)=0$ and
\begin{equation}  \begin{split}
  &{\widehat{V}}_{n,m}^{\theta}(t,x) 
  =  
  \sum_{i=1}^{m^n} \tfrac{g\bigl(x+\widehat{W}^{(\theta,0,-i)}_1-\widehat{W}^{(\theta,0,-i)}_t\bigr)-g(x)}{m^n}  
   \tfrac{ 
  \widehat{W}^{(\theta, 0, -i)}_{1}- \widehat{W}^{(\theta, 0, -i)}_{t}
  }{ 1 - t }
  \\
  &+\sum_{\ell =0}^{n-1}\sum_{i=1}^{m^{n-\ell }}
   \tfrac{
\left[\left( f\circ 
  \widehat{V}_{\ell ,m}^{(\theta,\ell ,i)}\right)-\1_{\N}(\ell )\left (f\circ  \widehat{V}_{\ell -1,m}^{(\theta,-\ell ,i)}\right)\right]\!
  \left(\mathcal{R}^{(\theta, \ell ,i)}_t,x+\widehat{W}_{\mathcal{R}^{(\theta, \ell ,i)}_t}^{(\theta,\ell ,i)}-\widehat{W}_t^{(\theta,\ell ,i)}\right)
  }
  {m^{n-\ell }\varrho(t,\mathcal{R}^{(\theta, \ell ,i)}_t)}
   \tfrac{ 
  \widehat{W}_{\mathcal{R}^{(\theta, \ell ,i)}_t}^{(\theta,\ell ,i)}- \widehat{W}^{(\theta, \ell , i)}_{t}
  }{ \mathcal{R}^{(\theta, \ell ,i)}_t-t}.
\end{split} \label{d03}    \end{equation}
Next, the fact that
$\forall\,x\in\R^d\colon g(x)=g(-x)$ (see~\eqref{t19}), the fact that
$\forall\,x\in\R^d\colon f(x)=f(-x)$ (see~\eqref{t19}),
 and the fact that 
$\forall\,\theta\in\Theta\colon \widehat{W}^\theta=-W^\theta$ show that for all
$k\in\N_0$, $\theta\in\Theta$, 
$ (t_j)_{j\in[0,k]\cap\Z}\in [0,1)^{k+1} $
with $\forall\, j\in[0,k)\colon t_j<t_{j+1}$
it holds
that
\begin{align}\begin{split}
&
\left[
g\bigl(W^0_{t_k} + W^{\theta}_1-W^{\theta}_{t_k}\bigr)-g(W_{t_k}^0)\right]
\tfrac{  W^{\theta}_{1}- W^{\theta}_{t_k}}{1-t_k}
\prod_{j=1}^{k} \tfrac{f\!\left({W}^0_{t_{j}}-{W}^0_{t_{j-1}}\right)}{\sqrt{\varrho(t_{j-1},t_j)}(t_{j}-t_{j-1})}\\
&
=-
\left[
g\bigl(\widehat{W}^0_{t_k} + \widehat{W}^{\theta}_1-\widehat{W}^{\theta}_{t_k}\bigr)-g(\widehat{W}_{t_k}^0)\right]
\tfrac{  \widehat{W}^{\theta}_{1}- \widehat{W}^{\theta}_{t_k}}{1-t_k}
\prod_{j=1}^{k} \tfrac{f\!\left(\widehat{W}^0_{t_{j}}-\widehat{W}^0_{t_{j-1}}\right)}{\sqrt{\varrho(t_{j-1},t_j)}(t_{j}-t_{j-1})}
.\end{split}
\end{align}
This and the fact that $(\widehat{W}^\theta)_{\theta\in\Theta}$ and 
 $({W}^\theta)_{\theta\in\Theta}$ are identically distributed
 show  that for all
$k\in\N_0$,
 $\theta\in\Theta$,
$ (t_j)_{j\in[0,k]\cap\Z}\in [0,1)^{k+1} $
with $\forall\, j\in[0,k)\colon t_j<t_{j+1}$
 it holds
that
\begin{align}
\E\!\left[
\left[
g\bigl(W^0_{t_k} + W^{\theta}_1-W^{\theta}_{t_k}\bigr)-g(W_{t_k}^0)\right]
\tfrac{  W^{\theta}_{1}- W^{\theta}_{t_k}}{1-t_k}
\prod_{j=1}^{k} \tfrac{f\!\left({W}^0_{t_{j}}-{W}^0_{t_{j-1}}\right)}{\sqrt{\varrho(t_{j-1},t_j)}(t_{j}-t_{j-1})}\right]
=0.\label{t11}
\end{align}
Next, we will prove
by induction on $n\in\N_0$
 that for all
$n\in\N_0$, $m\in\N$, $\theta\in \Theta$, 
$t\in[0,1)$, $x\in\R^d$
it holds that
${V}_{n,m}^\theta(t,x)=- \widehat{V}_{n,m}^\theta(t,-x)$. 
The fact that
$\forall\,m\in \N, \theta\in\Theta\colon V^\theta_{0,m}=
0= \widehat{V}^\theta_{0,m}$ shows the base case $n=0$.
For the induction step $\N_0\ni n-1\mapsto n\in\N$ let 
$n\in\N$ satisfy for all 
$\ell\in [0,n-1]\cap\Z$,
 $m\in\N$, $\theta\in \Theta$, 
$t\in[0,1)$, $x\in\R^d$
that
$\widehat{V}_{\ell,m}^\theta(t,x)=- {V}_{\ell,m}^\theta(t,-x)$. 
This, \eqref{c10}, 
the fact that
$\forall\,x\in\R^d\colon g(x)=g(-x)$ (see \eqref{t19}), the fact that
$\forall\,x\in\R^d\colon f(x)=f(-x)$ (see \eqref{t19}),
the fact that $\forall\,\theta\in\Theta\colon\widehat{W}^\theta=-W^\theta$, and 
 \eqref{d03}
 show for all $m\in\N$, 
$\theta\in \Theta$, 
$t\in[0,1)$, $x\in\R^d$  that
\begin{align}  
  &{V}_{n,m}^{\theta}
(t,x)
  =  
  \sum_{i=1}^{m^n} \tfrac{g\bigl(x+W^{(\theta,0,-i)}_1-W^{(\theta,0,-i)}_t\bigr)-g(x)}{m^n}     \tfrac{ 
  W^{(\theta, 0, -i)}_{1}- W^{(\theta, 0, -i)}_{t}
  }{ 1 - t }
\nonumber  \\
  & \quad+\sum_{\ell =0}^{n-1}\sum_{i=1}^{m^{n-\ell }}
   \tfrac{\left[
\left(f \circ 
  {V}_{\ell ,m}^{(\theta,\ell ,i)}\right)-\1_{\N}(\ell ) \left(f\circ {V}_{\ell -1,m}^{(\theta,-\ell ,i)}\right)\right]\!
  \left(\mathcal{R}^{(\theta, \ell ,i)}_t,\, x+W_{\mathcal{R}^{(\theta, \ell ,i)}_t}^{(\theta,\ell ,i)}-W_t^{(\theta,\ell ,i)}\right)
  }
  {m^{n-\ell }\varrho(t,\mathcal{R}^{(\theta, \ell ,i)}_t)}
   \tfrac{ 
  W_{\mathcal{R}^{(\theta, \ell ,i)}_t}^{(\theta,\ell ,i)}- W^{(\theta, \ell , i)}_{t}
  }{ \mathcal{R}^{(\theta, \ell ,i)}_t-t}\nonumber\\
&=
\sum_{i=1}^{m^n} \tfrac{g\bigl( -x+\widehat{W}^{(\theta,0,-i)}_1-\widehat{W}^{(\theta,0,-i)}_t\bigr)-g( -x)}{m^n}  
  \left[- \tfrac{ 
  \widehat{W}^{(\theta, 0, -i)}_{1}- \widehat{W}^{(\theta, 0, -i)}_{t}
  }{ 1 - t }\right]\nonumber
  \\
  &\quad +\sum_{\ell =0}^{n-1}\sum_{i=1}^{m^{n-\ell }}
   \tfrac{
\left[\left(f\circ
  {\widehat{V}}_{\ell ,m}^{(\theta,\ell ,i)}\right)-\1_{\N}(\ell ) \left(f\circ {\widehat{V}}_{\ell -1,m}^{(\theta,-\ell ,i)}\right)\right]\!
  \left(\mathcal{R}^{(\theta, \ell ,i)}_t,\, -x+\widehat{W}_{\mathcal{R}^{(\theta, \ell ,i)}_t}^{(\theta,\ell ,i)}-\widehat{W}_t^{(\theta,\ell ,i)}\right)
  }
  {m^{n-\ell }\varrho(t,\mathcal{R}^{(\theta, \ell ,i)}_t)}
\left[-
   \tfrac{ 
  \widehat{W}_{\mathcal{R}^{(\theta, \ell ,i)}_t}^{(\theta,\ell ,i)}- \widehat{W}^{(\theta, \ell , i)}_{t}
  }{ \mathcal{R}^{(\theta, \ell ,i)}_t-t}\right]\nonumber\\
&=-\widehat{V}_{n,m}^{\theta}(t,-x).
   \end{align}
This completes the induction step. Induction hence proves
for all
$n\in\N_0$, $m\in\N$, $\theta\in \Theta$, 
$t\in[0,1)$, $x\in\R^d$ that
\begin{align}\label{d05}
{V}_{n,m}^\theta
(t,x)
=- \widehat{V}_{n,m}^\theta(t,-x). 
\end{align}
This and the fact that
$\forall\, x \in\R^d\colon  
f(x)=f(-x)$ (see \eqref{t19}) show that for all
$k,\ell\in\N_0$,
$m\in\N$, $\nu,\eta\in \Theta$, $ (t_j)_{j\in[0,k]\cap\Z}\in [0,1)^{k+1} $
with $\forall\, j\in[0,k)\colon t_j<t_{j+1}$ it holds that
\begin{align}\begin{split}
&
   \tfrac{
\left(f\circ 
  {V}_{\ell ,m}^{\nu}\right)
  \bigl(\uniform_{t_k}^\eta,W_{t_k}^0+W_{\uniform_{t_k}^\eta}^{\eta}-W_{t_k}^{\eta}\bigr)
  }
  {\varrho({t_k},\uniform_{t_k}^\eta)}
   \tfrac{ 
  W_{\uniform_{t_k}^\eta}^{\eta}- W^{\eta}_{t_k}
  }{ \uniform_{t_k}^\eta-{t_k}}
\prod_{j=1}^{k} \tfrac{f\!\left(W^0_{t_{j}}-W^0_{t_{j-1}}\right)}{\sqrt{\varrho(t_{j-1},t_j)}(t_{j}-t_{j-1})}
\\
&
=-
  \tfrac{
\left(f\circ
  \widehat{V}_{\ell ,m}^{\nu}\right)
  \bigl(\uniform_{t_k}^\eta,\widehat{W}_{t_k}^0+\widehat{W}_{\uniform_{t_k}^\eta}^{\eta}-\widehat{W}_{t_k}^{\eta}\bigr)
  }
  {\varrho({t_k},\uniform_{t_k}^\eta)}
   \tfrac{ \widehat{W}_{\uniform_{t_k}^\eta}^{\eta}- \widehat{W}^{\eta}_{t_k}
  }{ \uniform_{t_k}^\eta-{t_k}}
\prod_{j=1}^{k} \tfrac{f\!\left(\widehat{W}^0_{t_{j}}-\widehat{W}^0_{t_{j-1}}\right)}{\sqrt{\varrho(t_{j-1},t_j)}(t_{j}-t_{j-1})}
.\label{d06}\end{split}
\end{align}
This and the fact that $(V_{m,n}^\theta, W^\theta,
\unif^\theta  )_{\theta\in\Theta,m\in\N,n\in\N_0}$ and 
$(\widehat{V}_{m,n}^\theta, \widehat{W}^\theta,\unif^\theta )_{\theta\in\Theta,m\in\N,n\in\N_0}$ are identically distributed
show
(for existence of the expectation cf.\ \cite[Lemma 3.3]{hutzenthaler2021overcoming})
 that for all
 $k,\ell\in\N_0$,
$m\in\N$, $\nu,\eta\in \Theta$, $ (t_j)_{j\in[0,k]\cap\Z}\in [0,1)^{k+1} $
with $\forall\, j\in[0,k)\colon t_j<t_{j+1}$ it holds that
\begin{align}\label{d08}
\E\!\left[  \tfrac{
\left(f\circ 
  {V}_{\ell ,m}^{\nu}\right)
  \bigl(\uniform_{t_k}^\eta,W_{t_k}^0+W_{\uniform_{t_k}^\eta}^{\eta}-W_{t_k}^{\eta}\bigr)
  }
  {\varrho({t_k},\uniform_{t_k}^\eta)}
   \tfrac{ 
  W_{\uniform_{t_k}^\eta}^{\eta}- W^{\eta}_{t_k}
  }{ \uniform_{t_k}^\eta-{t_k}}
\prod_{j=1}^{k} \tfrac{f\!\left(W^0_{t_{j}}-W^0_{t_{j-1}}\right)}{\sqrt{\varrho(t_{j-1},t_j)}(t_{j}-t_{j-1})}\right]
=0.
\end{align}
Next, \eqref{t19}, the fact that
$\forall\, s,t\in[0,1]\colon  \E\! \left[\lvert W_{1,s}^0-W_{1,t}^0 \rvert^2\right]=\lvert s-t\rvert $, and the fact that
$\forall\, t\in[0,1), s\in (t,1]\colon \varrho (t,s)(s-t)= \frac{1}{2\sqrt{1-t}\sqrt{s-t}}(s-t)= \frac{\sqrt{s-t}}{2\sqrt{1-t}}$
show for all $t\in[0,1)$, $s\in (t,1]$
that
\begin{align}
\begin{split}
&
\E\!\left[
\left\lvert
 \tfrac{f(W_{ s}^0-W^0_{ t})}{\sqrt{\varrho( t, s)}( s- t)}
\right\rvert^2\right]= 
\sum_{i=2}^{d}
 \tfrac{\E\!\left[\left\lvert W_{i, s}^0-W^0_{i, t}\right\rvert^2\right]}{\varrho( t, s)( s- t)^2}
=  
 \tfrac{(d-1)( s- t)}{\varrho( t, s)( s- t)^2}
=  \tfrac{d-1}{\varrho( t, s)( s- t)}= 
2 (d-1) \tfrac{\sqrt{1-t}}{\sqrt{s-t}}\geq 2(d-1)
.
\end{split}\label{t04}
\end{align}
This and the fact that
Brownian motions have independent increments imply that for all
$ (t_j)_{j\in[0,k]\cap\Z}\in [0,1)^{k+1} $
with $\forall\, j\in[0,k)\colon t_j<t_{j+1}$
it holds that
\begin{align}
\E\!\left[
\prod_{j=1}^{k}\left\lvert
\tfrac{f(W_{t_j}^0-W^0_{t_{j-1}})}{\sqrt{\varrho(t_{j-1},t_j)}(t_{j}-t_{j-1})}
\right\rvert^2\right]=
\prod_{j=1}^{k}
\E\!\left[\left\lvert
\tfrac{f(W_{t_j}^0-W^0_{t_{j-1}})}{\sqrt{\varrho(t_{j-1},t_j)}(t_{j}-t_{j-1})}
\right\rvert^2\right]=2^k(d-1)^{k}\prod_{j=1}^{k}\tfrac{\sqrt{1-t_{j-1}}}{\sqrt{t_j-t_{j-1}}}. \label{g01}
\end{align}
Furthermore,  \eqref{t19}, the triangle inequality,
the scaling property of Brownian motions, 
independence of Brownian coordinate processes,
 the fact that
$\E[\lvert W_{1,1}^0\rvert^2 ]=1 $, and
the fact that
$\E\!\left[\lvert f( W^0_1) \rvert^2\right]=\sum_{j=2}^{d}\E[\lvert W_{j,1}^0\rvert^2 ]=d-1$
 show that 
for all 
$t\in [0,1)$ it holds a.s.\ that
\begin{align}\begin{split}
&
\left(
\E \!\left[
\left \lvert 
\tfrac{\left(
g(W_{t}^0+W_1^0-W^0_{t})-g(W_{t}^0)\right) 
f\bigl(  W^{0}_{1}- W^{0}_{t}\bigr)}{ 1 - t }
\right\rvert^2\middle|\F_t\right]\right)^{1/2}\\
&\leq 
\left(\E\!\left[
\left \lvert  \left(W^{0}_{1,1}-W^{0}_{1,t}\right)
\tfrac{ f(
  W^{0}_{1}- W^{0}_{t})}{ 1 - t }
\right \rvert^2\middle|\F_t\right]\right)^{1/2}=
\left(\E\!\left[
\left \lvert  \tfrac{W^{0}_{1,1}-W^{0}_{1,t}}{\sqrt{1-t}}
\tfrac{ f(
  W^{0}_{1}- W^{0}_{t})}{ \sqrt{1 - t }}
\right \rvert^2\right]\right)^{1/2}\\
&=
\left(\E\!\left[\lvert  W^0_{1,1} \rvert^2\right]\right)^{\frac{1}{2}}\left(\E\!\left[\lvert f( W^0_1) \rvert^2\right]\right)^{\frac{1}{2}}=\sqrt{d-1}\leq \sqrt{d}.\end{split}
\end{align}
This, the triangle inequality, the tower property, the Markov property of Brownian motions, and
\eqref{g01}
show that 
for all $n,m\in\N$,
$k\in\N_0$, 
$ (t_j)_{j\in[0,k]\cap\Z}\in [0,1)^{k+1} $
with $\forall\, j\in[0,k)\colon t_j<t_{j+1}$ it holds that
\begin{align}\label{t21}\begin{split}
&
\left(\E\!\left[
\left\lvert
\sum_{i=1}^{m^n}\tfrac{g\bigl(W^0_{t_k}+W^{(0,0,-i)}_1-W^{(0,0,-i)}_{t_k}\bigr)-g(W^0_{t_k})}{m^n}  
  \tfrac{ f\bigl(
  W^{(0, 0, -i)}_{1}- W^{(0, 0, -i)}_{t_k}\bigr)
  }{ 1 - t_k }
\prod_{j=1}^{k}
\tfrac{f(W_{t_j}^0-W^0_{t_{j-1}})}{\sqrt{\varrho(t_{j-1},t_j)}(t_{j}-t_{j-1})}
\right\rvert^2
\right]\right)^{1/2}\\
&\leq  \left(\E\!\left[\E\!\left[
\left\lvert
  \tfrac{\left[g(W^0_{t_k}+W^{0}_1-W^{0}_{t_k})-g(W^0_{t_k})\right] f\bigl(
  W^{0}_{1}- W^{0}_{t_k}\bigr)
  }{ 1 - t_k }\right\rvert^2\middle|\F_{t_k}
\right]
\prod_{j=1}^{k}\left\lvert
\tfrac{f(W_{t_j}^0-W^0_{t_{j-1}})}{\sqrt{\varrho(t_{j-1},t_j)}(t_{j}-t_{j-1})}\right\rvert^2
\right]
\right)^{1/2}\\
&\leq d^{1/2}
 \left(\E\!\left[
\prod_{j=1}^{k}\left\lvert
\tfrac{f(W_{t_j}^0-W^0_{t_{j-1}})}{\sqrt{\varrho(t_{j-1},t_j)}(t_{j}-t_{j-1})}\right\rvert^2
\right]
\right)^{1/2}
= d^{1/2}\left(2^k(d-1)^{k}\prod_{j=1}^{k}\tfrac{\sqrt{1-t_{j-1}}}{\sqrt{t_j-t_{j-1}}}
\right)^{1/2}\\
&\leq  
  d^{\frac{k+1}{2}}
2^{\frac{k}{2}}
\left(\prod_{j=1}^{k}\tfrac{\sqrt{1-t_{j-1}}}{\sqrt{t_j-t_{j-1}}}
\right)^{1/2}.
\end{split}\end{align}
Next, \eqref{d08},
\cref{t01},
%
 independence and distributional properties of MLP approximations (see \cite[Lemma~3.2]{hutzenthaler2021overcoming}), 
the fact that Brownian motions have independent increments,
the disintegration theorem (see, e.g., \cite[Lemma 2.2]{hutzenthaler2020overcoming}),
the fact that
$\forall\,t\in [0,1), h\in C([t,1],[0,\infty))\colon \E \!\left[h(\uniform_t^0)\right]=\int_{t}^{1}h(s)\varrho(t,s)\,ds  $ 
(see \cref{d14}),
 and \eqref{t19}
show that for all $n,m\in\N$,
$\ell\in \{0,1,\ldots,n-1\}$, 
$k\in\N_0$, 
$ (t_j)_{j\in[0,k]\cap\Z}\in [0,1)^{k+1} $
with $\forall\, j\in[0,k)\colon t_j<t_{j+1}$ it holds that
\begin{align}
&\sum_{\nu=2}^{d}
\E \! \left[\left\lvert\sum_{i=1}^{m^{n-\ell }}
 \tfrac{
\left( f\circ 
{V}_{\ell ,m}^{(0,\ell ,i)}\right)
\!
  \bigl(\uniform^{(0, \ell ,i)}_{t_k},W^0_{t_k}+W_{\uniform^{(0, \ell ,i)}_{t_k}}^{(0,\ell ,i)}-W_{t_k}^{(0,\ell ,i)}\bigr)
  }
  {m^{n-\ell }\varrho(t_k,\uniform^{(0, \ell ,i)}_{t_k})}
  \tfrac{ 
  W_{\nu,\uniform^{(0, \ell ,i)}_{t_k}}^{(0,\ell ,i)}- W^{(0, \ell , i)}_{\nu,t_k}
  }{ \uniform^{(0, \ell ,i)}_{t_k}-t_k}
\prod_{j=1}^{k}
\tfrac{f(W_{t_j}^0-W^0_{t_{j-1}})}{\sqrt{\varrho(t_{j-1},t_j)}(t_{j}-t_{j-1})}
\right\rvert^2
\right]\nonumber\\
&= \sum_{\nu=2}^{d}
\var\! \left[\sum_{i=1}^{m^{n-\ell }}
 \tfrac{
\left( f\circ 
{V}_{\ell ,m}^{(0,\ell ,i)}\right)
\!
  \bigl(\uniform^{(0, \ell ,i)}_{t_k},W^0_{t_k}+W_{\uniform^{(0, \ell ,i)}_{t_k}}^{(0,\ell ,i)}-W_{t_k}^{(0,\ell ,i)}\bigr)
  }
  {m^{n-\ell }\varrho(t_k,\uniform^{(0, \ell ,i)}_{t_k})}
  \tfrac{ 
  W_{\nu,\uniform^{(0, \ell ,i)}_{t_k}}^{(0,\ell ,i)}- W^{(0, \ell , i)}_{\nu,t_k}
  }{ \uniform^{(0, \ell ,i)}_{t_k}-t_k}
\prod_{j=1}^{k}
\tfrac{f(W_{t_j}^0-W^0_{t_{j-1}})}{\sqrt{\varrho(t_{j-1},t_j)}(t_{j}-t_{j-1})}\right]\nonumber\\
&\geq \tfrac{m^{n-\ell }}{(m^{n-\ell })^2}\sum_{\nu=2}^{d} 
\var\! \left[
 \tfrac{
\left( f\circ 
{V}_{\ell ,m}^{(0,\ell ,1)}\right)
\!
  \bigl(\uniform^{(0, \ell ,1)}_{t_k},W^0_{t_k}+W_{\uniform^{(0, \ell ,1)}_{t_k}}^{(0,\ell ,1)}-W_{t_k}^{(0,\ell ,1)}\bigr)
  }
  {\varrho(t_k,\uniform^{(0, \ell ,1)}_{t_k})}
  \tfrac{ 
  W_{\nu,\uniform^{(0, \ell ,1)}_{t_k}}^{(0,\ell ,1)}- W^{(0, \ell , 1)}_{\nu,t_k}
  }{ \uniform^{(0, \ell ,1)}_{t_k}-t_k}
\prod_{j=1}^{k}
\tfrac{f(W_{t_j}^0-W^0_{t_{j-1}})}{\sqrt{\varrho(t_{j-1},t_j)}(t_{j}-t_{j-1})}\right]\nonumber\\
&= \tfrac{1}{m^{n-\ell }}\sum_{\nu=2}^{d}
\E\! \left[\left\lvert
 \tfrac{
\left( f\circ 
{V}_{\ell ,m}^{(0,\ell ,1)}\right)
\!
  \bigl(\uniform^{(0, \ell ,1)}_{t_k},W^0_{t_k}+W_{\uniform^{(0, \ell ,1)}_{t_k}}^{(0,\ell ,1)}-W_{t_k}^{(0,\ell ,1)}\bigr)
  }
  {\varrho(t_k,\uniform^{(0, \ell ,1)}_{t_k})}
  \tfrac{ 
  W_{\nu,\uniform^{(0, \ell ,1)}_{t_k}}^{(0,\ell ,1)}- W^{(0, \ell , 1)}_{\nu,t_k}
  }{ \uniform^{(0, \ell ,1)}_{t_k}-t_k}\right\rvert^2
\prod_{j=1}^{k}\left\lvert
\tfrac{f(W_{t_j}^0-W^0_{t_{j-1}})}{\sqrt{\varrho(t_{j-1},t_j)}(t_{j}-t_{j-1})}
\right\rvert^2\right]\nonumber\\
&= \tfrac{1}{m^{n-\ell }}\sum_{\nu=2}^{d}
\E\! \left[\left\lvert
 \tfrac{
\left( f\circ 
{V}_{\ell ,m}^{0}\right)
\!
  \bigl(\uniform^{0}_{t_k},W^0_{t_k}+W_{\uniform^{0}_{t_k}}^{0}-W_{t_k}^{0}\bigr)
  }
  {\varrho(t_k,\uniform^{0}_{t_k})}
  \tfrac{
  W_{\nu,\uniform^{0}_{t_k}}^{0}- W^{0}_{\nu,t_k}
  }{ \uniform^{0}_{t_k}-t_k}\right\rvert^2
\prod_{j=1}^{k}\left\lvert
\tfrac{f(W_{t_j}^0-W^0_{t_{j-1}})}{\sqrt{\varrho(t_{j-1},t_j)}(t_{j}-t_{j-1})}
\right\rvert^2\right]\nonumber\\
&= \tfrac{1}{m^{n-\ell }}
\E\! \left[\left\lvert
 \tfrac{
\left( f\circ 
{V}_{\ell ,m}^{0}\right)
\!
  \bigl(\uniform^{0}_{t_k},W_{\uniform^{0}_{t_k}}^{0}\bigr)
  }
  {\varrho(t_k,\uniform^{0}_{t_k})}
  \tfrac{ f\bigl(
  W_{\uniform^{0}_{t_k}}^{0}- W^{0}_{t_k}\bigr)
  }{ \uniform^{0}_{t_k}-t_k}\right\rvert^2
\prod_{j=1}^{k}\left\lvert
\tfrac{f(W_{t_j}^0-W^0_{t_{j-1}})}{\sqrt{\varrho(t_{j-1},t_j)}(t_{j}-t_{j-1})}
\right\rvert^2\right]\nonumber\\
&= \tfrac{1}{m^{n-\ell}}
\int_{t_k}^{1}\E\left[\left \lvert
  \tfrac{(f\circ  {V}_{\ell,m}^{0})\!
  \left(t_{k+1},W_{t_{k+1}}^{0}\right)
  }
  {\varrho(t_k,t_{k+1})}
\tfrac{f\bigl(
W_{t_{k+1}}^0-W^0_{t_{k}}\bigr)}{t_{k+1}-t_{k}}
 \prod_{j=1}^{k}
\tfrac{f(
W_{t_j}^0-W^0_{t_{j-1}})}{\sqrt{\varrho(t_{j-1},t_j)}(t_{j}-t_{j-1})
}\right \rvert^2\right]
\varrho(t_k,t_{k+1})\,
dt_{k+1}\nonumber
\\
&= \tfrac{1}{m^{n-\ell}}
\int_{t_k}^{1}\E\!\left[\left \lvert
(f\circ {V}_{\ell,m}^{0})\bigl(t_{k+1},W_{t_{k+1}}^{0}\bigr)
 \prod_{j=1}^{k+1}
\tfrac{f(
W_{t_j}^0-W^0_{t_{j-1}})}{\sqrt{\varrho(t_{j-1},t_j)}(t_{j}-t_{j-1})
}
\right \rvert^2\right]
dt_{k+1}.\label{t20}
\end{align}
Next, 
\eqref{t19} shows that
$\lvert f(W_1^0)\rvert^2= \sum_{j=2}^{d}\lvert W_{j,1}^0\rvert^2$ is chi-square distributed with $(d-1)$ degrees of freedom. Hence, 
$\E [\lvert f(W_1^0) \rvert^4]= (d-1)(d+1)\leq d^2 $. This,
 independence and distributional properties of MLP approximations (see \cite[Lemma~3.2]{hutzenthaler2021overcoming}),
  \cref{m01} (applied with
 $\gamma \gets (\R^d\ni (x_1,\ldots,x_d)\mapsto \lvert x_1\rvert\in\R)$  in the notation of \cref{m01}), and the fact that
$\E[\lvert W_{1,1}^0\rvert^4 ]=3 $
 show that for all
$m,n\in \N$,
$t\in [0,1)$ it holds a.s.\
that
\begin{align}
&
\left(
\E\!\left[
\left\lvert
\left(f
\circ 
V_{n,m}^0\right)(t, W^0_{t})
\right\rvert^2\middle|
\F_t\right]
\right)^{1/2}
=
\left(
\E\!\left[
\left\lvert
\left(f
\circ 
V_{n,m}^0\right)(t, \xi)
\right\rvert^2\right]\bigr|_{\xi=W_t^0}
\right)^{1/2}
\nonumber\\
&\leq \left(\E\!\left[\lvert  W^0_{1,1} \rvert^4\right]\right)^{1/4}
 \max\!\left\{\left(\E\!\left[\lvert f( W^0_1) \rvert^4\right]\right)^{{n}/{4}},1\right\}6^{n-1}\leq 3^{1/4}d^{n/2}6^{n-1}\leq 
d^{n/2}6^n
.\label{t22}
\end{align}
This, the triangle inequality, independence and distributional properties of MLP approximations (see \cite[Lemma~3.2]{hutzenthaler2021overcoming}),  
the fact that Brownian motions have independent increments, the disintegration theorem (see, e.g., \cite[Lemma~2.2]{hutzenthaler2020overcoming}),
the fact that
$\forall\,t\in [0,1), h\in C([t,1],[0,\infty))\colon \E \!\left[h(\uniform_t^0)\right]=\int_{t}^{1}h(s)\varrho(t,s)\,ds  $ 
(see \cref{d14}), and \eqref{g01}
show that for all $n,m\in\N$,
$\ell\in \{0,1,\ldots,n-1\}$, 
$\lambda\in \{-\ell,\ell\}$, $\iota\in\{0,1\}$,
$k\in\N_0$, 
$ (t_j)_{j\in[0,k]\cap\Z}\in [0,1)^{k+1} $
with $\forall\, j\in[0,k)\colon t_j<t_{j+1}$ it holds that
\begin{align}\label{t24}
&\left(\E\!\left[
\left\lvert
\sum_{i=1}^{m^{n-\ell }}
 \tfrac{
 \left(f\circ  {V}_{\ell -\iota,m}^{(0,\lambda ,i)}\right)
  \bigl(\uniform^{(0, \ell ,i)}_{t_k},W^0_{t_k}+W_{\uniform^{(0, \ell ,i)}_{t_k}}^{(0,\ell ,i)}-W_{t_k}^{(0,\ell ,i)}\bigr)
  }
  {m^{n-\ell }\varrho(t_k,\uniform^{(0, \ell ,i)}_{t_k})}
  \tfrac{ f\bigl(
  W_{\uniform^{(0, \ell ,i)}_{t_k}}^{(0,\ell ,i)}- W^{(0, \ell , i)}_{t_k}
\bigr)
  }{ \uniform^{(0, \ell ,i)}_{t_k}-t_k}\prod_{j=1}^{k}
\tfrac{f(W_{t_j}^0-W^0_{t_{j-1}})}{\sqrt{\varrho(t_{j-1},t_j)}(t_{j}-t_{j-1})}
\right\rvert^2
\right]\right)^{1/2}\nonumber\\
&\leq 
\left(\E\!\left[
\left\lvert
 \tfrac{
\left(f\circ  {V}_{\ell -\iota,m}^{0}\right)
  \bigl(\uniform^{0}_{t_k},W^0_{t_k}+W_{\uniform^{0}_{t_k}}^{0}-W_{t_k}^{0}\bigr)
  }
  {\varrho(t_k,\uniform^{0}_{t_k})}
  \tfrac{ f\bigl(
  W_{\uniform^{0}_{t_k}}^{0}- W^{0}_{t_k}\bigr)
  }{ \uniform^{0}_{t_k}-t_k}\prod_{j=1}^{k}
\tfrac{f(W_{t_j}^0-W^0_{t_{j-1}})}{\sqrt{\varrho(t_{j-1},t_j)}(t_{j}-t_{j-1})}
\right\rvert^2
\right]\right)^{1/2}\nonumber\\
&= \left(
\int_{t_k}^{1}\E\!\left[\left \lvert
(f\circ {V}_{\ell-\iota,m}^{0})\bigl(t_{k+1},W_{t_{k+1}}^{0}\bigr)
 \prod_{j=1}^{k+1}
\tfrac{f(
W_{t_j}^0-W^0_{t_{j-1}})}{\sqrt{\varrho(t_{j-1},t_j)}(t_{j}-t_{j-1})
}
\right \rvert^2\right]
dt_{k+1}\right)^{1/2}\nonumber\\
&= \left(
\int_{t_k}^{1}\E\!\left[
\E\!\left[
\left \lvert
(f\circ {V}_{\ell-\iota,m}^{0})\bigl(t_{k+1},W_{t_{k+1}}^{0}\bigr)\right\rvert^2
\middle|\F_{t_{k+1}}
\right]
 \prod_{j=1}^{k+1}\left\lvert
\tfrac{f(
W_{t_j}^0-W^0_{t_{j-1}})}{\sqrt{\varrho(t_{j-1},t_j)}(t_{j}-t_{j-1})
}
\right \rvert^2\right]
dt_{k+1}\right)^{1/2}\nonumber\\
&\leq d^{\frac{\ell-\iota}{2}}6^{\ell-\iota} \left(
\int_{t_k}^{1}\E\!\left[
 \prod_{j=1}^{k+1}\left\lvert
\tfrac{f(
W_{t_j}^0-W^0_{t_{j-1}})}{\sqrt{\varrho(t_{j-1},t_j)}(t_{j}-t_{j-1})
}
\right \rvert^2\right]
dt_{k+1}\right)^{1/2}\nonumber\\
&\leq d^{\frac{\ell}{2}}6^{\ell}
2^{\frac{k+1}{2}}d^{\frac{k+1}{2}}\left[
\int_{t_k}^{1}
\prod_{j=1}^{k+1}\tfrac{\sqrt{1-t_{j-1}}}{\sqrt{t_j-t_{j-1}}}
dt_{k+1}\right]^{1/2}= d^{\frac{k+1+\ell}{2}}2^{\frac{k+1}{2}+\ell}3^{\ell}
\left[
\int_{t_k}^{1}
\prod_{j=1}^{k+1}\tfrac{\sqrt{1-t_{j-1}}}{\sqrt{t_j-t_{j-1}}}
dt_{k+1}\right]^{1/2}.
\end{align}
Next, observe for all
$x,y\in\R$ that
$
(x+y)^2=   x^2+y^2+2 \frac{x}{\sqrt{2}}\sqrt{2}y
\geq  x^2+y^2- \frac{x^2}{2} -2 y^2=
\frac{x^2}{2}-y^2.
$
This shows for all 
$x,y\in\R$ that
$(x+y)^2+y^2\geq \frac{x^2}{2}$.
This, \eqref{t19}, and the triangle inequality show that for all random variables $X=(X_1,X_2,\ldots,X_d),Y=(Y_1,Y_2,\ldots,Y_d)\colon \Omega\to\R^d$ with 
$\sum_{\nu=1}^{d}\E [\lvert X_\nu\rvert^2+\lvert Y_\nu\rvert^2]<\infty$ it holds
 that
\begin{align}
&
\left(\E \!\left[\lvert f(X+Y)\rvert^2\right]\right)^{1/2}
=\left(\sum_{\nu=2}^{d} \E\!\left[(X_\nu+Y_\nu)^2\right] \right)^{1/2}
+\left(\sum_{\nu=2}^{d} \E\!\left[\lvert Y_\nu\rvert^2\right] \right)^{1/2}
-\left(\sum_{\nu=2}^{d} \E\!\left[\lvert Y_\nu\rvert^2\right] \right)^{1/2}\nonumber\\
&\geq 
\left(\sum_{\nu=2}^{d} \E\!\left[(X_\nu+Y_\nu)^2\right] + \E\!\left[\lvert Y_\nu\rvert^2\right] \right)^{1/2}
-\left(\sum_{\nu=2}^{d} \E\!\left[\lvert Y_\nu\rvert^2\right] \right)^{1/2}\nonumber\\
&\geq 
\left(\sum_{\nu=2}^{d} \frac{\E\!\left[\lvert X_\nu\rvert^2\right]}{2} \right)^{1/2}-\left(\sum_{\nu=2}^{d} \E\!\left[\lvert Y_\nu\rvert^2\right] \right)^{1/2}=
\left(\frac{1}{2}\sum_{\nu=2}^{d}  \E\!\left[\lvert X_\nu\rvert^2\right]\right)^{1/2}-
\left( \E \bigl[\lvert f(Y)\rvert^2\bigr]\right)^{1/2}.\label{t17}
\end{align}
Furthermore, \eqref{c10} and the fact that $\forall\, m\in\N, \theta\in\Theta\colon f\circ V_{0,m}^\theta=0$ imply for all
$  n,m \in \N$, 
$ t\in [0,1)$,
$x \in \R^d$
that
\begin{equation}  \begin{split}
  &{V}_{n,m}^{0}(t,x) 
=\left[
\sum_{i=1}^{m^{n-\ell }}
 \tfrac{
\left( f\circ 
  {V}_{\ell ,m}^{(0,\ell ,i)}\right)
\!
  \left(\uniform^{(0, \ell ,i)}_t,x+W_{\uniform^{(0, \ell ,i)}_t}^{(0,\ell ,i)}-W_t^{(0,\ell ,i)}\right)
  }
  {m^{n-\ell }\varrho(t,\uniform^{(0, \ell ,i)}_t)}
  \tfrac{ 
  W_{\uniform^{(0, \ell ,i)}_t}^{(0,\ell ,i)}- W^{(0, \ell , i)}_{t}
  }{ \uniform^{(0, \ell ,i)}_t-t}\right]\Biggr|_{\ell=n-1}\\
&\qquad + \sum_{i=1}^{m^n}\tfrac{g\bigl(x+W^{(0,0,-i)}_1-W^{(0,0,-i)}_t\bigr)-g(x)}{m^n}  
  \tfrac{ 
  W^{(0, 0, -i)}_{1}- W^{(0, 0, -i)}_{t}
  }{ 1 - t }\\
&\qquad-
\left[\sum_{i=1}^{m^{n-\ell }}
 \tfrac{
\1_{\N}(\ell ) \left(f\circ  {V}_{\ell -1,m}^{(0,-\ell ,i)}\right)
\!
  \left(\uniform^{(0, \ell ,i)}_t,x+W_{\uniform^{(0, \ell ,i)}_t}^{(0,\ell ,i)}-W_t^{(0,\ell ,i)}\right)
  }
  {m^{n-\ell }\varrho(t,\uniform^{(0, \ell ,i)}_t)}
  \tfrac{ 
  W_{\uniform^{(0, \ell ,i)}_t}^{(0,\ell ,i)}- W^{(0, \ell , i)}_{t}
  }{ \uniform^{(0, \ell ,i)}_t-t}\right]\Biggr|_{\ell=n-1}
  \\ &\qquad+\sum_{\ell =1}^{n-2}\sum_{i=1}^{m^{n-\ell }}
 \tfrac{
\left[\left( f\circ 
  {V}_{\ell ,m}^{(0,\ell ,i)}\right)-\1_{\N}(\ell ) \left(f\circ  {V}_{\ell -1,m}^{(0,-\ell ,i)}\right)\right]
\!
  \left(\uniform^{(0, \ell ,i)}_t,x+W_{\uniform^{(0, \ell ,i)}_t}^{(0,\ell ,i)}-W_t^{(0,\ell ,i)}\right)
  }
  {m^{n-\ell }\varrho(t,\uniform^{(0, \ell ,i)}_t)}
  \tfrac{ 
  W_{\uniform^{(0, \ell ,i)}_t}^{(0,\ell ,i)}- W^{(0, \ell , i)}_{t}
  }{ \uniform^{(0, \ell ,i)}_t-t}.
\end{split}     \end{equation}
This, \eqref{t17}, \eqref{t20}, \eqref{t21}, \eqref{t24}, and the triangle inequality (recall \eqref{t19}) show that for all
$  n,m \in \N$,
$k\in\N_0$,
$ (t_j)_{j\in[0,k]\cap\Z}\in [0,1)^{k+1} $
with $\forall\, j\in[0,k)\colon t_j<t_{j+1}$ it holds that
\begin{align}
&\left(\E\! \left[\left\lvert
\bigl(f\circ {V}_{n,m}^{0})(t_k,W^0_{t_k}) \right\rvert^2
\prod_{j=1}^{k}\left\lvert
\tfrac{f(W_{t_j}^0-W^0_{t_{j-1}})}{\sqrt{\varrho(t_{j-1},t_j)}(t_{j}-t_{j-1})}
\right\rvert^2
\right]\right)^{1/2}\nonumber\\
&\geq \vastleft{25pt}(\tfrac{1}{2}\sum_{\nu=2}^{d}
\E  \vastleft{25pt}[\vastleft{25pt}\lvert\sum_{i=1}^{m^{n-\ell }}
 \tfrac{
\left( f\circ 
{V}_{\ell ,m}^{(0,\ell ,i)}\right)
\!
  \bigl(\uniform^{(0, \ell ,i)}_{t_k},W^0_{t_k}+W_{\uniform^{(0, \ell ,i)}_{t_k}}^{(0,\ell ,i)}-W_{t_k}^{(0,\ell ,i)}\bigr)
  }
  {m^{n-\ell }\varrho(t_k,\uniform^{(0, \ell ,i)}_{t_k})}
\tfrac{
  W_{\nu,\uniform^{(0, \ell ,i)}_{t_k}}^{(0,\ell ,i)}- W^{(0, \ell , i)}_{\nu,t_k}
  }{ \uniform^{(0, \ell ,i)}_{t_k}-t_k}\nonumber\\
&\qquad\qquad\qquad\qquad\qquad\qquad\qquad\qquad\qquad\qquad\qquad
\prod_{j=1}^{k}
\tfrac{f(W_{t_j}^0-W^0_{t_{j-1}})}{\sqrt{\varrho(t_{j-1},t_j)}(t_{j}-t_{j-1})}
\vastright{25pt}\rvert^2
\vastright{25pt}]
 \vastright{25pt})^{1/2}\Biggr|_{\ell=n-1}\nonumber\\
&-\left(\E\!\left[
\left\lvert
\sum_{i=1}^{m^n}\tfrac{g\bigl(W^0_{t_k}+W^{(0,0,-i)}_1-W^{(0,0,-i)}_{t_k}\bigr)-g(W^0_{t_k})}{m^n}  
  \tfrac{ f\bigl(
  W^{(0, 0, -i)}_{1}- W^{(0, 0, -i)}_{t_k}\bigr)
  }{ 1 - t_k }
\prod_{j=1}^{k}
\tfrac{f(W_{t_j}^0-W^0_{t_{j-1}})}{\sqrt{\varrho(t_{j-1},t_j)}(t_{j}-t_{j-1})}
\right\rvert^2
\right]\right)^{1/2}\nonumber\\
&-\sum_{\ell=1}^{n-2}
\vastleft{25pt}(
\E\! \vastleft{25pt}[
\vastleft{25pt}\lvert
\sum_{i=1}^{m^{n-\ell }}
 \tfrac{
 \left(f\circ  {V}_{\ell ,m}^{(0,\ell ,i)}\right)
  \bigl(\uniform^{(0, \ell ,i)}_{t_k},W^0_{t_k}+W_{\uniform^{(0, \ell ,i)}_{t_k}}^{(0,\ell ,i)}-W_{t_k}^{(0,\ell ,i)}\bigr)
  }
  {m^{n-\ell }\varrho(t_k,\uniform^{(0, \ell ,i)}_{t_k})}
  \tfrac{ f\bigl(
  W_{\uniform^{(0, \ell ,i)}_{t_k}}^{(0,\ell ,i)}- W^{(0, \ell , i)}_{t_k}
\bigr)  }{ \uniform^{(0, \ell ,i)}_{t_k}-t_k}
\nonumber\\
&\qquad\qquad\qquad\qquad\qquad\qquad\qquad\qquad\qquad\qquad\qquad
\prod_{j=1}^{k}
\tfrac{f(W_{t_j}^0-W^0_{t_{j-1}})}{\sqrt{\varrho(t_{j-1},t_j)}(t_{j}-t_{j-1})}
\vastright{25pt}\rvert^2
\vastright{25pt}]\vastright{25pt})^{1/2}\nonumber\\
&-\sum_{\ell=1}^{n-1}\vastleft{25pt}(\E\!\vastleft{25pt}[\vastleft{25pt}\lvert
\sum_{i=1}^{m^{n-\ell }}
 \tfrac{
\left(f\circ  {V}_{\ell -1,m}^{(0,-\ell ,i)}\right)
  \bigl(\uniform^{(0, \ell ,i)}_{t_k},W^0_{t_k}+W_{\uniform^{(0, \ell ,i)}_{t_k}}^{(0,\ell ,i)}-W_{t_k}^{(0,\ell ,i)}\bigr)
  }
  {m^{n-\ell }\varrho(t_k,\uniform^{(0, \ell ,i)}_{t_k})}
  \tfrac{ f\bigl(
  W_{\uniform^{(0, \ell ,i)}_{t_k}}^{(0,\ell ,i)}- W^{(0, \ell , i)}_{t_k}
\bigr)
  }{ \uniform^{(0, \ell ,i)}_{t_k}-t_k}\nonumber\\
&\qquad\qquad\qquad\qquad\qquad\qquad\qquad\qquad\qquad\qquad\qquad
\prod_{j=1}^{k}
\tfrac{f(W_{t_j}^0-W^0_{t_{j-1}})}{\sqrt{\varrho(t_{j-1},t_j)}(t_{j}-t_{j-1})}
\vastright{25pt}\rvert^2
\vastright{25pt}]\vastright{25pt})^{1/2}\nonumber\\
&\geq \left(\tfrac{1}{2m}
\int_{t_k}^{1}\E\!\left[\left \lvert
(f\circ {V}_{n-1,m}^{0})\bigl(t_{k+1},W_{t_{k+1}}^{0}\bigr)
 \prod_{j=1}^{k+1}
\tfrac{f(
W_{t_j}^0-W^0_{t_{j-1}})}{\sqrt{\varrho(t_{j-1},t_j)}(t_{j}-t_{j-1})
}
\right \rvert^2\right]
dt_{k+1}\right)^{1/2}\nonumber\\
&\quad - d^{\frac{k+1}{2}}
2^{\frac{k}{2}}
\left(\prod_{j=1}^{k}\tfrac{\sqrt{1-t_{j-1}}}{\sqrt{t_j-t_{j-1}}}
\right)^{1/2}-2\sum_{\ell=1}^{n-2}\left[ d^{\frac{k+1+\ell}{2}}2^{\frac{k+1}{2}+\ell}3^{\ell}\left[
\int_{t_k}^{1}
\prod_{j=1}^{k+1}\tfrac{\sqrt{1-t_{j-1}}}{\sqrt{t_j-t_{j-1}}}
dt_{k+1}\right]^{1/2}\right].\label{k01}
\end{align}
Next, the triangle inequality implies that for all measure spaces $(E,\mathcal{E},\nu)$, all $n\in\N$, and all 
 measurable functions $h_0,h_1,\ldots,h_n\colon E\to [0,\infty)$
with $\sum_{\ell=0}^{n}\int_{E}\lvert h_\ell\rvert^2\,d\nu <\infty$
and $h_0\geq h_1-\sum_{\ell=2}^{n}h_\ell $
it holds that
$(\int_{E}\lvert h_0\rvert^2\,d\nu)^{1/2}\geq (\int_{E}\lvert h_1\rvert^2\,d\nu)^{1/2}
-\sum_{\ell=2}^{n}(\int_{E}\lvert h_\ell\rvert^2\,d\nu)^{1/2}
$.
This, \eqref{c06d}, \eqref{c06c}, \eqref{k01},  and \cref{c02} imply for all $m,n\in\N$, $k\in\N_0$ that
\begin{align}\begin{split}
&\epsDown_{k,m}(n)\geq \frac{1}{\sqrt{2m}}
\epsDown_{k+1,m}(n-1)
- d^{\frac{k+1}{2}}2^{\frac{k}{2}} 2^{\frac{k}{2}}
-2\sum_{\ell=1}^{n-2}\left[ d^{\frac{k+1+\ell}{2}}2^{\frac{k+1}{2}+\ell}3^{\ell}2^{\frac{k+1}{2}}\right]
\\
&\geq \frac{1}{\sqrt{2m}}
\epsDown_{k+1,m}(n-1)
-2d^{\frac{k+1+(n-2)}{2}}
2^{k+1+(n-2)}
\sum_{\ell=0}^{n-2}3^{\ell}
\\
&=
\frac{1}{\sqrt{2m}}
\epsDown_{k+1,m}(n-1)
-2d^{\frac{k+n-1}{2}}
2^{k+n-1}\frac{3^{n-1}-1}{2}\geq 
\frac{1}{\sqrt{2m}}
\epsDown_{k+1,m}(n-1)
-d^{\frac{k+n-1}{2}}
2^{k}6^{n-1}.
\end{split}\label{t30}\end{align}
This shows
for all $m,n\in\N$,  $\ell\in \{0,1,\ldots,n-1\}$  that
\begin{align}
&
\left(\tfrac{1}{\sqrt{2m}}\right)^\ell\epsDown_{\ell,m}(n-\ell)
-\sum_{j=0}^{\ell-1}\left[
\left(\tfrac{1}{\sqrt{2m}}\right)^jd^{\frac{n-1}{2}}
2^{j}6^{n-j-1} \right]\nonumber\\
&
\geq 
\left(\tfrac{1}{\sqrt{2m}}\right)^\ell
\left[\tfrac{1}{\sqrt{2m}}
\epsDown_{\ell+1,m}(n-\ell-1)
-
d^{\frac{\ell+n-\ell-1}{2}}2^{\ell}6^{n-\ell-1}
\right]
-\sum_{j=0}^{\ell-1}\left[
\left(\tfrac{1}{\sqrt{2m}}\right)^jd^{\frac{n-1}{2}}
2^{j}6^{n-j-1} \right]
\nonumber \\
&=\left(\tfrac{1}{\sqrt{2m}}\right)^{\ell+1}
\epsDown_{\ell+1,m}(n-\ell-1)-
\sum_{j=0}^{\ell}\left[
\left(\tfrac{1}{\sqrt{2m}}\right)^jd^{\frac{n-1}{2}}
2^{j}6^{n-j-1} \right].\label{t31}
\end{align}
This and induction  imply for all $m,n\in\N$ that
\begin{align}\begin{split}
&\epsDown_{0,m}(n)\geq 
\left(\tfrac{1}{\sqrt{2m}}\right)^{n-1}\epsDown_{n-1,m}(1)
-\sum_{j=0}^{n-2}\left[
\left(\tfrac{1}{\sqrt{2m}}\right)^jd^{\frac{n-1}{2}}
2^{j}6^{n-j-1} \right]\\
&= \left(\tfrac{1}{\sqrt{2m}}\right)^{n-1}\epsDown_{n-1,m}(1)-
d^{\frac{n-1}{2}}6^{n-1}
\sum_{j=0}^{n-2}\left(\tfrac{1}{3\sqrt{2m}}\right)^j
\geq \left(\tfrac{1}{\sqrt{2m}}\right)^{n-1}\epsDown_{n-1,m}(1)-
d^{\frac{n-1}{2}}6^{n-1}
\sum_{j=0}^{\infty}\left(\tfrac{1}{3}\right)^j
\\
&
=  \left(\tfrac{1}{\sqrt{2m}}\right)^{n-1}
\epsDown_{n-1,m}(1)-
d^{\frac{n-1}{2}}6^{n}\frac{1}{6(1-\frac{1}{3})}
\geq
\left(\tfrac{1}{\sqrt{2m}}\right)^{n-1}
\epsDown_{n-1,m}(1)-d^{\frac{n-1}{2}}6^n.
\end{split}\label{t31b}\end{align}
Furthermore, \eqref{t19}, the fact that $\forall\,B\in \mathcal{B}(\R^d),s\in[0,1]\colon \P (W_s\in B)= \P (-W_s\in B)$, the triangle inequality, 
the fact that the coordinates of Brownian motions are i.i.d.,
 Jensen's inequality, and the fact that
$\forall\,s\in[0,1]\colon \E \bigl[\lvert W_{1,s}^0\rvert^2\bigr]=s$
 show for all $x=(x_1,x_2,\ldots,x_d)\in\R^d$, $s\in[0,1]$ that
\begin{align}\begin{split}
&
\E \!\left[ (
g(x+W_s^0)-g(x))^2\left\lvert f( W_{s}^0)\right\rvert^2
\right]
= \sum_{j=2}^{d} \E \!\left[\Bigl(
\left\lvert x_1 +W_{1,s}^0\right\rvert -\left\lvert x_1 \right\rvert  \Bigr)^2
\left\lvert W_{j,s}^0\right\rvert^2
\right]\\
&= \sum_{j=2}^{d}\E \!\left[
\Bigl(
\left\lvert x_1 +W_{1,s}^0\right\rvert^2 +\left\lvert x_1 \right\rvert^2-2  \left\lvert x_1 +W_{1,s}^0\right\rvert \left\lvert x_1 \right\rvert \Bigr)
\right]
\E\!\left[\left\lvert W_{j,s}^0\right\rvert^2\right]
\\
&\geq (d-1)
\E \!\left[
\Bigl(
\left\lvert x_1\right\rvert^2 + \left\lvert W_{1,s}^0\right\rvert^2
+2x_1W_{1,s}^0 +\left\lvert x_1 \right\rvert^2-
2  \left(\lvert x_1\rvert +\lvert W_{1,s}^0\rvert \right)
\left\lvert x_1\right\rvert 
\Bigr)
\right]
\E\!\left[\left\lvert W_{1,s}^0\right\rvert^2\right]\\
&= (d-1)
\E \!\left[
 \left\lvert W_{1,s}^0\right\rvert^2
-
2  \left\lvert W_{1,s}^0\right\rvert
\left\lvert x_1\right\rvert 
\right]\E\!\left[
\left\lvert W_{1,s}^0\right\rvert^2\right]\\
&\geq (d-1)
 \biggl[
\E \!\left[
 \left\lvert W_{1,s}^0\right\rvert^2\right]
-
2 \left(\E\!\left[ \left\lvert W_{1,s}^0\right\rvert^2\right]\right)^{1/2}
\left\lvert x_1\right\rvert 
\biggr]
\E\!\left[
\left\lvert W_{1,s}^0\right\rvert^2\right]\\
&=   (d-1)
\left(s-2\sqrt{s}\lvert x_1\rvert\right)s
=(d-1) \left(s^2-2\lvert x_1\rvert s^{3/2}\right).
\end{split}\label{t07}\end{align}
Next, Jensen's inequality, the fact that
$\forall\,t\in [0,1]\colon \E\! \left[\lvert W_{1,t}^0 \rvert^2\right]=t $, and the fact that
$[0,1)\ni t\mapsto\frac{t}{1-t}\in\R$ is 
non-decreasing show for all
 $t\in[0,\frac{1}{17}]$ that
\begin{align}
1- \frac{2\E\! \left[\lvert W_{1,t}^0 \rvert\right]
}{\sqrt{1-t}}\geq 1- \frac{2\left(\E\! \left[\lvert W_{1,t}^0 \rvert^2\right]\right)^{1/2}
}{\sqrt{1-t}}=1-\frac{2\sqrt{t}}{\sqrt{1-t}}
\geq \frac{1}{2}.
\end{align}
This, the tower property, the Markov property, \eqref{t07}, 
the fact that
$(f(W^0_t))_{t\in[0,1]} $ and $ (W_{1,t}^0)_{t\in[0,1]}$ are independent 
(recall \eqref{t19}),
the fact that Brownian motions have independent increments, and \eqref{t04}
show that for all
$k\in\N_0$, 
$ (t_j)_{j\in[0,k]\cap\Z}\in [0,\frac{1}{17}]^{k+1} $
with $\forall\, j\in[0,k)\colon t_j<t_{j+1}$
 it holds that
\begin{align}
&
\E \!\left[\left\lvert\left[
 g\bigl( W_{t_k}^0+W^{\theta}_1-W^{\theta}_{t_k}\bigr)-g( W_{t_k}^0)\right]
  \frac{ f\bigl(
  W^{\theta}_{1}- W^{\theta}_{t_k}\bigr)
  }{ 1 - t_k }
\prod_{j=1}^{k}
\frac{f(W_{t_j}^0-W^0_{t_{j-1}})}{\sqrt{\varrho(t_{j-1},t_j)}(t_{j}-t_{j-1})}
\right\rvert^2
\right]\nonumber\\
&=
\E\!\left[ \E \!\left[\left\lvert\left[
 g\bigl( W_{t_k}^0+W^{\theta}_1-W^{\theta}_{t_k}\bigr)-g( W_{t_k}^0)\right]
  \frac{ f\bigl(
  W^{\theta}_{1}- W^{\theta}_{t_k}\bigr)
  }{ 1 - t_k }
\right\rvert^2\middle|\F_{t_k}
\right]
\prod_{j=1}^{k}\left\lvert
\frac{f(W_{t_j}^0-W^0_{t_{j-1}})}{\sqrt{\varrho(t_{j-1},t_j)}(t_{j}-t_{j-1})}
\right\rvert^2
\right]\nonumber\\
&
=
\E\!\left[ \E \!\left[\left\lvert\left[
 g\bigl( x+W^{0}_{1-t_k}\bigr)-g( x)\right]
  \frac{ 
 f\bigl( W^{0}_{1-t_k}\bigr)
  }{ 1 - t_k }
\right\rvert^2
\right]\Bigr|_{x=W_{t_k}^0}
\prod_{j=1}^{k}\left\lvert
\frac{f(W_{t_j}^0-W^0_{t_{j-1}})}{\sqrt{\varrho(t_{j-1},t_j)}(t_{j}-t_{j-1})}
\right\rvert^2
\right]\nonumber\\
&\geq \E\!\left[ (d-1)
\frac{(1-t_k)^2-2\left\lvert W_{1,t_k}^0 \right\rvert(1-t_k)^{3/2}}{(1-t_k)^2}
\prod_{j=1}^{k}\left\lvert
\frac{f(W_{t_j}^0-W^0_{t_{j-1}})}{\sqrt{\varrho(t_{j-1},t_j)}(t_{j}-t_{j-1})}
\right\rvert^2
\right]\nonumber\\
&= (d-1)\E\!\left[ \left(1-
\frac{2\left\lvert W_{1,t_k}^0 \right\rvert}{\sqrt{1-t_k}}\right)
\prod_{j=1}^{k}\left\lvert
\frac{f(W_{t_j}^0-W^0_{t_{j-1}})}{\sqrt{\varrho(t_{j-1},t_j)}(t_{j}-t_{j-1})}
\right\rvert^2
\right]\nonumber\\
&= (d-1)\left[
1- \frac{2\E\! \left[\lvert W_{1,t_k}^0 \rvert\right]
}{\sqrt{1-t_k}}\right]
\prod_{j=1}^{k}\E\!\left[\left\lvert
\frac{f(W_{t_j}^0-W^0_{t_{j-1}})}{\sqrt{\varrho(t_{j-1},t_j)}(t_{j}-t_{j-1})}
\right\rvert^2\right]\nonumber\\
&\geq (d-1) \frac{1}{2}2^k(d-1)^{k}=(d-1)^{1+k} 2^{k-1}.
\label{t05}
\end{align}
Next, \eqref{c10} and the fact that $f(0)=0$ (see \eqref{t19}) show for all
$m\in\N$, $t\in[0,1)$, $x\in\R^d$ that
\begin{equation}  \begin{split}
  &{V}_{1,m}^{0}(t,x) 
  =  
  \sum_{i=1}^{m} \tfrac{g\bigl(x+W^{(0,0,-i)}_1-W^{(0,0,-i)}_t\bigr)-g(x)}{m}  
   \tfrac{ 
  W^{(0, 0, -i)}_{1}- W^{(0, 0, -i)}_{t}
  }{ 1 - t }.
\end{split}    \end{equation}
This, \eqref{t19}, \eqref{t11}, 
\cref{t01},
 independence and distributional properties of MLP approximations
(see \cite[Lemma~3.2]{hutzenthaler2021overcoming}), and \eqref{t05}
show that for all $k\in\N_0$,
$\color{blue}m\in\N$,
$ (t_j)_{j\in[0,k]\cap\Z}\in [0,\frac{1}{17}]^{k+1} $
with $\forall\, j\in[0,k)\colon t_j<t_{j+1}$ it holds
that
\begin{align}
&
\E\!\left[
\left\lvert
\left(f
\circ 
V_{1,m}^0\right)(t_k, W^0_{t_k})
\prod_{j=1}^{k}
\tfrac{f(W_{t_j}^0-W^0_{t_{j-1}})}{\sqrt{\varrho(t_{j-1},t_j)}(t_{j}-t_{j-1})}
\right\rvert^2\right]\nonumber\\
&=
\E\!\left[\sum_{\nu=2}^{d}
\left\lvert\sum_{i=1}^{m} \tfrac{\left[g\bigl(W_{t_k}^0+W^{(0,0,-i)}_1-W^{(0,0,-i)}_{t_k}\bigr)-g(W_{t_k}^0)\right]}{m}  
   \tfrac{ 
  W^{(0, 0, -i)}_{\nu,1}- W^{(0, 0, -i)}_{\nu,t_k}
  }{ 1 - t_k }\right\rvert^2
\prod_{j=1}^{k}\left\lvert
\tfrac{f(W_{t_j}^0-W^0_{t_{j-1}})}{\sqrt{\varrho(t_{j-1},t_j)}(t_{j}-t_{j-1})}
\right\rvert^2\right]\nonumber\\
&=\sum_{\nu=2}^{d}
\var\!\left[
\sum_{i=1}^{m} \tfrac{\left[g\bigl(W_{t_k}^0+W^{(0,0,-i)}_1-W^{(0,0,-i)}_{t_k}\bigr)-g(W_{t_k}^0)\right]}{m}  
   \tfrac{ W^{(0, 0, -i)}_{\nu,1}- W^{(0, 0, -i)}_{\nu,t_k}  }{ 1 - t_k }
\prod_{j=1}^{k}
\tfrac{f(W_{t_j}^0-W^0_{t_{j-1}})}{\sqrt{\varrho(t_{j-1},t_j)}(t_{j}-t_{j-1})}
\right]\nonumber\\
&\geq 
\frac{m}{m^2}\sum_{\nu=2}^{d}
\var\!\left[
\left[g\bigl(W_{t_k}^0+W^{(0,0,-1)}_1-W^{(0,0,-1)}_{t_k}\bigr)-g(W_{t_k}^0)\right]
   \tfrac{ W^{(0, 0, -1)}_{\nu,1}- W^{(0, 0, -1)}_{\nu,t_k}  }{ 1 - t_k }
\prod_{j=1}^{k}
\tfrac{f(W_{t_j}^0-W^0_{t_{j-1}})}{\sqrt{\varrho(t_{j-1},t_j)}(t_{j}-t_{j-1})}
\right]\nonumber
\\
&= \frac{1}{m}\sum_{\nu=2}^{d}
\E\!\left[
\left\lvert \left[g\bigl(W_{t_k}^0+W^{(0,0,-1)}_1-W^{(0,0,-1)}_{t_k}\bigr)-g(W_{t_k}^0)\right]
   \tfrac{ 
  W^{(0, 0, -1)}_{\nu,1}- W^{(0, 0, -1)}_{\nu,t_k}
  }{ 1 - t_k }\right\rvert^2
\prod_{j=1}^{k}\left\lvert
\tfrac{f(W_{t_j}^0-W^0_{t_{j-1}})}{\sqrt{\varrho(t_{j-1},t_j)}(t_{j}-t_{j-1})}
\right\rvert^2\right]\nonumber\\
&= \frac{1}{m}
\E\!\left[
\left\lvert \left[g\bigl(W_{t_k}^0+W^{(0,0,-1)}_1-W^{(0,0,-1)}_{t_k}\bigr)-g(W_{t_k}^0)\right]
   \tfrac{ f\bigl(
  W^{(0, 0, -1)}_{1}- W^{(0, 0, -1)}_{t_k}\bigr)
  }{ 1 - t_k }\right\rvert^2
\prod_{j=1}^{k}\left\lvert
\tfrac{f(W_{t_j}^0-W^0_{t_{j-1}})}{\sqrt{\varrho(t_{j-1},t_j)}(t_{j}-t_{j-1})}
\right\rvert^2\right]\nonumber\\
&\geq\frac{1}{m}(d-1)^{1+k} 2^{k-1}.
\end{align}
This, \eqref{c06d}, \eqref{c06c}, and the fact that
$\forall\,t_0\in\{0\},\delta\in [0,1],k\in\N\colon \int_{t_0}^\delta\int_{t_1}^{\delta}\ldots\int_{t_{k-1}}^{\delta}
dt_{k}dt_{k-1}\cdots dt_{1}=\frac{\delta^k}{k!}$ show for all
$t_0\in \{0\}$, $k,m\in\N$, $ \delta\in (0,\frac{1}{17}]$ 
that
$\epsDown_{0,m}(1)\geq [\frac{1}{m}(d-1)\frac{1}{2}]^{1/2}$
and
\begin{align}\label{t13}
&
\epsDown_{k,m}(1)\geq 
\left(\int_{t_0}^\delta\int_{t_1}^{\delta}\ldots\int_{t_{k-1}}^{\delta}
\E\!\left[
\left\lvert
\left(f
\circ 
V_{1,m}^0\right)(t_k, W^0_{t_k})
\prod_{j=1}^{k}
\tfrac{f(W_{t_j}^0-W^0_{t_{j-1}})}{\sqrt{\varrho(t_{j-1},t_j)}(t_{j}-t_{j-1})}
\right\rvert^2\right]
dt_{k}dt_{k-1}\cdots dt_{1}
\right)^{1/2}\nonumber\\
&\geq\left[ \frac{(d-1)^{1+k} 2^{k-1}}{m}\frac{\delta^k}{k!}\right]^{1/2}
=\frac{1}{2\sqrt{m}}(2(d-1))^{\frac{1+k}{2}}\frac{\delta^{k/2}}{\sqrt{k!}}
\geq 
\frac{1}{2\sqrt{m}}d^{\frac{1+k}{2}}\frac{\delta^{k/2}}{\sqrt{k!}}
\1_{[2,\infty)}(d)
.
\end{align}
This, \eqref{c06d}, and \eqref{t31b} show for all
$m,n\in\N$,
$ \delta\in (0,\frac{1}{17}]$ that
\begin{align}\begin{split}
&\left(
\E\!\left[
\left\lvert
\left(f
\circ 
V_{n,m}^0\right)(0, 0)
\right\rvert^2
\right]
\right)^{1/2}=\epsDown_{0,m}(n)
\geq 
\left(\tfrac{1}{\sqrt{2m}}\right)^{n-1}\epsDown_{n-1,m}(1)-d^{\frac{n-1}{2}}6^n
\\
&\geq \tfrac{1}{\sqrt{2}}
\left(\tfrac{1}{\sqrt{2m}}\right)^{n}d^{\frac{n}{2}}
\tfrac{\delta^{\frac{n-1}{2}}}{\sqrt{(n-1)!}}
1_{[2,\infty)}(d)
-d^{\frac{n-1}{2}}6^n.
\end{split}\label{t26}\end{align}
Furthermore, observe that for all 
$n,m\in\N$ with
$d\geq 8(2m)^n 17^{n-1} (n-1)!\cdot 36^n$ it holds that
\begin{align}
d^{1/2}\geq 2\sqrt{2}(\sqrt{2m})^n 17^{\frac{n-1}{2}}\sqrt{(n-1)!}\cdot  6^n
\end{align} and
hence
\begin{align}
\tfrac{1}{2\sqrt{2}}
\left(\tfrac{1}{\sqrt{2m}}\right)^{n}d^{\frac{n}{2}}
\tfrac{\left(\frac{1}{17}\right)^{\frac{n-1}{2}}}{\sqrt{(n-1)!}}
\geq d^{\frac{n-1}{2}}6^n
.
\end{align}
This and \eqref{t26} show that  for all 
$n,m\in\N$ with
$d\geq (1224m)^n n!$ it holds that
\begin{align}
&\left(
\E\!\left[
\left\lvert
\left(f
\circ 
V_{n,m}^0\right)(0, 0)
\right\rvert^2
\right]
\right)^{1/2}\geq
\tfrac{1}{\sqrt{8}}
\left(\tfrac{1}{\sqrt{2m}}\right)^{n}d^{\frac{n}{2}}
\tfrac{\left(\frac{1}{17}\right)^{\frac{n-1}{2}}}{\sqrt{(n-1)!}}
\geq
\tfrac{
\left(
\frac{1}{2m}\right)^{\frac{n}{2}}
d^{\frac{n}{2}}
\left(\frac{1}{17}\right)^{\frac{n}{2}}}{\sqrt{n!}}=
\tfrac{d^{n/2}}{(34m)^{n/2}  \sqrt{n!}}
.
\end{align}
This completes the proof of \cref{d01}.
\end{proof}

Finally, since our nonlinearity is bounded by the squared Euclidean norm,
\cref{d01} immediately implies a lower bound for the approixmation errors of
MLP approximations.
\begin{corollary}\label{c12}
Assume \cref{s01}, 
let $\lVert\cdot \rVert\colon\R^d\to[0,\infty)$ be the standard norm,
let
$u \colon  [0,1] \times \R^d \to\R $, and
assume
for all $t\in[0,1]$, $x=(x_1,x_2,\ldots,x_d)\in\R^d$ that
\begin{align}
g(x)=\lvert x_1\rvert, \quad f(x)=\bigl({\textstyle\sum_{j=2}^{d}}\lvert x_j\rvert^2\bigr)^{1/2},\quad\text{and}\quad 
u(t,x)= \frac{1}{(2\pi )^{d/2}}
\int_{\R^d}g\bigl(x+z\sqrt{1-t}\bigr)e^{- \frac{\lVert z \rVert^2}{2}}\,dz.\label{k05}
\end{align}
Then 
\begin{enumerate}[i)]
\item\label{c12a} it holds for all $t\in[0,1)$, $x=(x_1,x_2,\ldots, x_d)\in\R^d$ that
$u|_{[0,1)\times \R^d}\in C^{1,2}([0,1)\times \R^d,\R)$,
$
\bigl( \tfrac{ \partial }{ \partial t } u \bigr)( t, x ) + \tfrac{ 1 }{ 2 } ( \Delta_x u )( t, x ) + (f\circ \nabla_x u)  (t, x)  = 0$, and
$
u(1,x)=g(x)
$, and
\item it holds 
for all 
$n,m\in\N$
 with
$d\geq (1224m)^n n!$ it holds that
\begin{align}\begin{split}
&
\tfrac{d^{n/2}}{(34m)^{n/2}  \sqrt{n!}}-1\\
&
\leq 
\left(
\E\!\left[
\left\lvert{U}_{n,m}^{0}(0,0)-u(0,0)\right \rvert^2\right]+
\E\!\left[
\left\lVert{V}_{n,m}^{0}(0,0)-(\nabla_x u)(0,0)\right \rVert^2\right]\right)^{1/2}\leq d^{\frac{n+1}{2}}6^{n+1}.
\end{split}\end{align}\end{enumerate}
\end{corollary}
\begin{proof}[Proof of \cref{c12}]Throughout the proof 
for every $i\in [1,d]\cap\N$ let $ W_{i,1}^0$ be the $i$-th coordinate of $W^0_1$.
First, \eqref{k05} and a standard result on (time reverse) linear heat equations show for all $t\in[0,1)$, $x\in\R^d$ that
\begin{align}\label{k14}
u|_{[0,1)\times \R^d}\in C^{1,2}([0,1)\times \R^d,\R),\quad
\bigl( \tfrac{ \partial }{ \partial t } u \bigr)( t, x ) + \tfrac{ 1 }{ 2 } ( \Delta_x u )( t, x ) =0,\quad\text{and}\quad u(1,x)=g(x).
\end{align} 
Furthermore, \eqref{k05}  
 implies
 for all 
 $t\in[0,1]$, $x=(x_1,x_2,\ldots,x_d),
y=(y_1,y_2,\ldots,y_d)
\in\R^d$ that 
$\lvert u(t,x)-u(t,y)\rvert\leq 
\frac{1}{(2\pi )^{d/2}}
\int_{\R^d}\lvert g(x+z\sqrt{t})- g(y+z\sqrt{t})\rvert e^{- \frac{\lVert z \rVert^2}{2}}\,dz\leq 
\lvert x_1-y_1\rvert$.
This shows for all $t\in[0,1]$, $x\in\R^d$, $i\in [2,d]\cap\N$
that $ \lvert (\tfrac{\partial}{\partial x_1}u)(t,x)\rvert\leq 1$ and 
$  (\tfrac{\partial}{\partial x_i}u)(t,x)=0$.
 This and \eqref{k14}
establish
\eqref{c12a}.

Next, \eqref{c12a},
 the triangle inequality, 
\cref{d01},
and \eqref{k05}
imply for all $m,n\in\N$ with $d\geq (1224m)^n n!$ that
\begin{align}\label{k04}\begin{split}
&\left(
\E\!\left[
\lVert
V_{n,m}^0(0, 0)
-(\nabla_x u)(0,0)
\rVert^2
\right]
\right)^{1/2}\geq \left(
\E\!\left[
\lVert
V_{n,m}^0(0, 0)
\rVert^2
\right]
\right)^{1/2}-1
\geq 
\tfrac{d^{n/2}}{(34m)^{n/2}  \sqrt{n!}}-1
.
\end{split}\end{align}
Next, 
\eqref{k05} shows that
$\lvert f(W_1^0)\rvert^2= \sum_{j=2}^{d}\lvert W_{j,1}^0\rvert^2$ is chi-square distributed with $(d-1)$ degrees of freedom. Hence, 
$\E [\lvert f(W_1^0) \rvert^4]= (d-1)(d+1)\leq d^2 $. This, \cref{m01} (applied with
 $\gamma \gets (\R^d\ni (x_1,\ldots,x_d)\mapsto \lvert x_1\rvert\in\R)$  in the notation of \cref{m01}), and the fact that
$\E[\lvert W_{1,1}^0\rvert^4 ]=3 $ show that
\begin{align}\begin{split}
&
\left(
\E\!\left[
\left\lvert{U}_{n,m}^{0}(0,0)\right \rvert^2\right]+
\E\!\left[
\left\lVert{V}_{n,m}^{0}(0,0)\right \rVert^2\right]\right)^{1/2}
\leq 
\left(\E\!\left[\lvert  W^0_{1,1} \rvert^4\right]\right)^{1/4}
 \max\left\{\left(\E\!\left[\lvert f( W^0_1)\rvert^4\right]\right)^{n/4},1\right\}6^n\sqrt{d}\\
&=3^{1/4}d^{n/2}6^n \sqrt{d}= 3^{1/4}d^{\frac{n+1}{2}}6^n
.
\end{split}\end{align}
This, the triangle inequality, the fact that
$ \forall\, t\in[0,1], x\in\R^d\colon  \lvert (\tfrac{\partial}{\partial x_1}u)(t,x)\rvert\leq 1$,  the fact that
$\forall\,i\in[2,d]\cap\N \colon  (\tfrac{\partial}{\partial x_i}u)(t,x)=0$,
 and \eqref{k05}
show for all
$m,n\in\N$ that
\begin{align}\begin{split}
&
\left(
\E\!\left[
\left\lvert{U}_{n,m}^{0}(0,0)-u(0,0)\right \rvert^2\right]+
\E\!\left[
\left\lVert{V}_{n,m}^{0}(0,0)-(\nabla_x u)(0,0)\right \rVert^2\right]\right)^{1/2}\\
&\leq \left(
\E\!\left[
\left\lvert{U}_{n,m}^{0}(0,0)\right \rvert^2\right]+
\E\!\left[
\left\lVert{V}_{n,m}^{0}(0,0)\right \rVert^2\right]\right)^{1/2}
+\left(
\lvert u(0,0)\rvert^2
+
\lVert(\nabla_x u)(0,0)\rVert^2\right)^{1/2}\\
&\leq  3^{1/4}d^{\frac{n+1}{2}}6^n+\sqrt{1+1}\leq d^{\frac{n+1}{2}}6^{n+1} .
\end{split}\end{align}
This and \eqref{k04} complete the proof of \cref{c12}.
\end{proof}

\section*{Acknowledgements}
The authors are indebted to Arnulf Jentzen for valuable discussion and to Thomas Kruse for pointing out that \eqref{eq:intro.PDE} is the HJB equation of the stochastic control problem
\eqref{eq:intro.stochastic.control}.
This work has been partially funded by the Deutsche Forschungsgemeinschaft (DFG, German 
Research Foundation) through the research grant HU1889/7-1.


\end{document}